\documentclass[12pt,reqno]{amsart} 

\makeatletter
\@namedef{subjclassname@2020}{\textup{2020} Mathematics Subject Classification}
\makeatother

\usepackage{amsmath,amssymb,amsthm,amscd}
\usepackage[T1]{fontenc}
\usepackage{blkarray} 

\usepackage{mathrsfs,euscript} 
\usepackage{graphicx} 
\usepackage[colorlinks=true,urlcolor=blue,citecolor=red,linkcolor=blue,hyperfootnotes=true]{hyperref} 

\let\mc\mathcal

\let\eu\EuScript
\let\nc\newcommand

\newtheorem{thm}{Theorem}[section]

\newtheorem{lem}[thm]{Lemma}
\newtheorem{prop}[thm]{Proposition}

\theoremstyle{definition}
\newtheorem{defn}[thm]{Definition}
\newtheorem{rem}[thm]{Remark}

\newtheorem{example}[thm]{Example}

\numberwithin{equation}{section}

\def\beq{\begin{equation}}
\def\eeq{\end{equation}}
\def\be{\begin{equation*}}
\def\ee{\end{equation*}}
\nc{\bea}{\begin{eqnarray*}}
\nc{\eea}{\end{eqnarray*}}

\let\al\alpha
\let\bt\beta
\let\dl\delta

\let\eps\varepsilon
\let\gm\gamma
\let\Gm\Gamma

\let\la\lambda
\let\La\Lambda

\let\phi\varphi
\let\si\sigma
\let\sig\varsigma

\let\thi\vartheta

\let\Ups\Upsilon
\let\om\omega
\let\Om\Omega

\let\der\partial
\let\Hat\widehat
\let\Tilde\widetilde
\let\bra\langle
\let\ket\rangle

\let\on\operatorname

\def\End{\on{End}}

\def\Hom{\on{Hom}}

\def\Res{\on{Res}}

\def\N{{\mathbb N}}
\def\C{{\mathbb C}}
\def\Z{{\mathbb Z}}
\def\Q{{\mathbb Q}}

\def\Pb{{\mathbb P}}
\def\R{{\mathbb R}}

\usepackage[OT2,T1]{fontenc}
\usepackage[all,cmtip]{xy} 
\usepackage{cancel}
\usepackage[vcentermath]{youngtab} 
\usepackage{empheq} 
\usepackage{changepage}
\usepackage{bm} 

\newlanguage\fakelanguage
\newcommand\cyr{\fontencoding{OT2}\fontfamily{wncyr}\selectfont
   \language\fakelanguage}
\DeclareTextFontCommand{\textcyr}{\cyr}

\usepackage{calligra}

\DeclareMathOperator{\HOM}{\mathscr{H}\text{\kern -3pt {\calligra\large om}}\,}

\makeatletter
\newsavebox{\@brx}
\newcommand{\llangle}[1][]{\savebox{\@brx}{$\m@th{#1\langle}$}%
  \mathopen{\copy\@brx\kern-0.5\wd\@brx\usebox{\@brx}}}
\newcommand{\rrangle}[1][]{\savebox{\@brx}{$\m@th{#1\rangle}$}%
  \mathclose{\copy\@brx\kern-0.5\wd\@brx\usebox{\@brx}}}
\makeatother

\let\bi\bibitem

\newcommand{\edc}{\widehat\nabla}
\newcommand{\rqh}{|_{\substack{{\bf Q}=1\\ \hbar=1}}}
\usepackage{scalerel,stackengine}  
\stackMath
\newcommand\reallywidehat[1]{%
\savestack{\tmpbox}{\stretchto{%
  \scaleto{%
    \scalerel*[\widthof{\ensuremath{#1}}]{\kern.1pt\mathchar"0362\kern.1pt}%
    {\rule{0ex}{\textheight}}
  }{\textheight}%
}{2.4ex}}%
\stackon[-6.9pt]{#1}{\tmpbox}%
}
\newcommand{\bigboxtimes}{%
  \mathop{%
    \raisebox{-.7ex}{\scalebox{1.6}{$\boxtimes$}}
  }\displaylimits
}

\begin{document}
\title[Mittag--Leffler functions, Borel--Laplace multitransforms, and QDE's]{Generalized Mittag--Leffler functions, Borel--Laplace multitransforms, and quantum differential equations of Fano varieties}
\author[Giordano Cotti]{Giordano Cotti$\>^{\circ}$}

{\let\thefootnote\relax
\footnotetext{\vskip5pt 
\noindent
$^\circ\>$\textit{ E-mail}:  giordano.cotti@tecnico.ulisboa.pt}}

\maketitle
\begin{center}
\textit{ 
$^\circ\>$Grupo de F\'isica Matem\'atica \\
Departamento de Matem\'atica, Instituto Superior Técnico\\
Av. Rovisco Pais, 1049-001 Lisboa, Portugal\/\\}
\end{center}
\vskip2cm

\begin{abstract}
This paper addresses the integration problem for the isomonodromic system of quantum differential equations associated with smooth projective Fano varieties.

We begin by introducing a class of multivariable, multivalued analytic functions of Mittag--Leffler type. We study their analytic properties and provide explicit Mellin--Barnes integral representations. For any smooth Fano toric variety, we prove that suitable integral linear combinations of these generalized Mittag--Leffler functions and their derivatives form bases of solutions for the corresponding quantum differential equation.

In the second part of the paper, we introduce two families of integral transforms—referred to as Borel–Laplace multitransforms—acting on tuples of functions. We show that the Laplace multitransform enables the reconstruction of the full solution space of the quantum differential equations for Fano complete intersections, starting from a basis of solutions associated with the ambient Fano variety. Dually, the Borel multitransform reconstructs the solution space for the quantum differential equations of Fano projective bundles of arbitrary rank from that of the base variety.

These results both systematize and generalize the research directions initiated in \cite{Cot22,Cot24a}. Applications include explicit examples involving Fano toric complete intersections, Fano bundles on toric varieties, and low-dimensional Fano varieties such as strict del Pezzo surfaces.
\end{abstract}

\vskip0,3cm
\begin{adjustwidth}{35pt}{35pt}
{\footnotesize {\it Key words:} Gromov--Witten invariants, quantum cohomology, quantum differential equations, integral transforms, Mittag--Leffler functions, Mellin--Barnes integrals}
\vskip2mm
\noindent
{\footnotesize {\it 2020 Mathematics Subject Classification:} Primary 53D45, 44A20, 44A30; Secondary: 33E12}
\end{adjustwidth}

\tableofcontents

\section{Introduction}

\noindent 1.1.\,\,\,Fano manifolds are, by definition, smooth projective varieties with ample anticanonical class\footnote{Throughout this work, all varieties are assumed to be smooth and defined over $\C$.}. Their importance is fundamental in the classification framework of the Minimal Model Program \cite{KM98,Deb01,Bir12}. A key geometric feature of Fano manifolds is their {\it abundance} of rational curves: not only does every Fano manifold contain rational curves \cite{Mor79}, but it is in fact rationally connected, meaning that any two general points can be joined by an irreducible rational curve \cite{KMM92}.

In light of this property, Fano varieties offer a particularly rich and compelling case study from the perspective of Gromov--Witten theory. Indeed, the original spirit of Gromov--Witten theory is to promote the enumerative counts of curves -- of arbitrary genus (not only rational ones), and subject to suitable incidence conditions -- into symplectic invariants of a Kähler manifold. These enumerative invariants can be assembled in various ways to produce rich structures, such as quantum cohomology and quantum differential equations. See e.g.\,\,\cite{KM94,RT95,Rua96,SMcD25}.

The quantum differential equation (abbreviated as the \emph{qDE}) of a smooth projective variety $X$ is a linear ordinary differential equation, depending on parameters, for a function taking values in the even-degree cohomology of $X$. Specifically, the qDE of $X$ takes the form
\beq\label{eqintro1}
\frac{d}{dz}\sig(z,\bm{t}) = \left( \mathcal{U}(\bm{t}) + \frac{1}{z} \mu(\bm{t}) \right) \sig(z,\bm{t}), \quad \sig \colon \widetilde{\mathbb{C}^*} \times M_X \to H^{\mathrm{ev}}(X, \mathbb{C}),
\eeq
where $\Tilde{\C^*}$ is the universal cover of $\C^*$, and $M_X$ is the domain of convergence of the genus-zero Gromov–Witten potential $F_0^X(\bm{t})$, a generating function that encodes the genus-zero Gromov–Witten invariants of $X$. The operator $\mathcal{U}(\bm{t})$ is defined in terms of third derivatives of this potential and captures the scaling structure and homogeneous symmetries underlying the genus-zero Gromov–Witten theory. The operator $\mu(\bm{t})$ encodes information about the nonzero even-degree Betti numbers of $X$. Precise definitions will be provided in the main body of the article.

\vskip2mm
\noindent 1.2.\,\,\,The study of the qDE lies at the core of the isomonodromic approach to quantum cohomology, initiated within the analytic theory of Frobenius manifolds by B.A.\,Dubrovin \cite{Dub96,Dub98,Dub99}. Indeed, the qDE of $X$ encodes deep structural information about the variety.

On one hand, it turns out that the monodromy and Stokes phenomena of the solutions to the qDE \eqref{eqintro1} remain constant under small variations of the deformation parameters $\bm{t}$. Consequently, the entire genus-zero Gromov–Witten theory of $X$ can be recovered from the data of the qDE alone. For details on the Riemann–Hilbert–Birkhoff approach to reconstructing the genus-zero potential $F_0^X(\bm{t})$, and thereby the Dubrovin–Frobenius structure on the quantum cohomology of $X$, we refer the reader to \cite{Dub96, Dub99, Cot21, Cot24a}.

On the other hand, when $X$ is (deformation equivalent to a) Fano, the quantum differential equation is believed to reflect aspects not only of the enumerative and symplectic geometry of $X$, but also (conjecturally) of its topology and complex structure. Such information becomes accessible through a detailed study of the asymptotic behavior and monodromy properties of its solutions; see \cite{Dub98, GGI16, Cot20, Cot22, CDG24} for developments in this direction.

The goal of this paper is to develop new analytic techniques for studying the qDE by constructing explicit integral representations of its solutions. These representations will be especially well-suited for the analysis of asymptotics, Stokes phenomena, and other analytic aspects, building on the approach initiated in \cite{Cot22, Cot24}. This goal will be achieved for Fano toric varieties, Fano complete intersections in Fano varieties, and Fano projective bundles. In the latter two cases, the results presented here generalize and systematize those obtained in \cite{Cot22, Cot24}.
\vskip2mm
\noindent 1.3.\,\,\,Rather than directly studying the system of differential equations \eqref{eqintro1}, following an idea introduced in \cite{Cot22, Cot24}, we focus instead on the space of \emph{master functions} of $X$. For a fixed point $\bm{t}_0 \in M_X$, a master function of $X$ at $\bm{t}_0$ is a holomorphic $\mathbb{C}$-valued function $\Phi \colon \widetilde{\mathbb{C}^*} \to \mathbb{C}$ of the form
\[
\Phi(z) = z^{-\frac{\dim_{\mathbb{C}} X}{2}} \int_X \sig(z, \bm{t}_0), \quad z \in \widetilde{\mathbb{C}^*},
\]
where $\sig$ is a solution of \eqref{eqintro1} specialized at $\bm{t} = \bm{t}_0$.

One of the main results of \cite{Cot22} shows that, for generic $\bm{t} \in M_X$, the space of master functions at $\bm{t}$ determines the full space of solutions of the qDE \eqref{eqintro1}. Here, the notion of genericity refers to the complement of the so-called \emph{$\mathcal{A}_\Lambda$-stratum}, an analytic subset of $M_X$ given as the zero locus of suitable holomorphic functions (see Section \ref{scyc} for precise definitions).

It turns out that, for $\bm{t}$ outside the $\mathcal{A}_\Lambda$-stratum, the system \eqref{eqintro1} is equivalent to a scalar differential equation for the master functions $\Phi$ of $X$ at $\bm{t}$. This equivalence essentially corresponds to choosing a cyclic vector, namely the unit $1 \in H^0(X, \mathbb{C})$. The reconstruction of solutions $\sig$ of the full system \eqref{eqintro1} from the master functions is achieved via algebraic operations and differentiations.

Accordingly, we regard the integration of the scalar quantum differential equation satisfied by the master functions as the central problem to be addressed.

This paper presents three main results:

\begin{enumerate}
    \item A description of a generating set for the space of master functions of any Fano toric variety $ X $ at the point $ 0 \in M_X $, expressed in terms of a family of special functions introduced here, which we refer to as \textit{generalized Mittag--Leffler functions}, see Theorem \ref{MTH0};
    
    \item A \textit{quantum Lefschetz-type} theorem, which shows how to reconstruct the space of master functions for a (broad class of) Fano complete intersections $ X $ in a product of Fano varieties $ (X_i)_{i=1}^h $, via a \textit{Laplace $ \eu{M} $-multitransform} of the master functions associated with each factor $ X_i $, see Theorem \ref{TH1};
    
    \item A \textit{quantum Leray--Hirsch-type} theorem, which provides a method to reconstruct the space of master functions for a (broad class of) Fano projective bundles $ X $ over a product of Fano varieties $ (X_i)_{i=1}^h $, using a \textit{Borel $ \eu{M} $-multitransform} of the master functions of each factor $ X_i $, see Theorem \ref{mt2}.
\end{enumerate}

\begin{rem}
The Borel--Laplace $\eu M$-multitransforms introduced here extend the $(\bm\alpha,\bm\beta)$-multitransforms studied in \cite{Cot22,Cot24}. In particular, Theorem~\ref{TH1} generalizes both Theorem~7.2.1 and Theorem~7.3.1 of \cite{Cot22}, which correspond to specific cases. Similarly, Theorem~\ref{mt2} extends Theorem~3.20 of \cite{Cot24} to the setting of arbitrary Fano projective bundles.
\end{rem}

\begin{rem}
In many examples of interest, the qDE~\eqref{eqintro1} is known at points $ \bm{t} $ lying in the $ \mathcal{A}_\Lambda $-stratum, and are therefore ``non-generic''. Although such cases are not treated in \cite{Cot22}, all examples studied so far indicate that the reconstruction of solutions to the qDE from the space of master functions can still be carried out through algebraic manipulations and differentiations, up to the choice of a finite number of arbitrary auxiliary integration constants.
Hence, the main results of this paper remain relevant even in these ``non-generic'' situations. A notable instance is the qDE associated with the complex Grassmannian $ G(2,4) $, thoroughly investigated in \cite{CDG20}; see also \cite[Sec.\,7.4]{Cot22}. Another example of this type is discussed in Section~\ref{sdP} of the present paper, where the qDE of the del Pezzo surface $ \mathrm{dP}_2 $ (the blow-up of $ \mathbb{P}^2 $ at two points) is explicitly integrated in terms of generalized Mittag--Leffler functions.
\end{rem}

\vskip1mm
\noindent 1.4.\,\,\,We now briefly present the main results of the paper. To this end, we first introduce the two main analytical tools.

The first is a family of special functions. In a brief note from 1903 \cite{ML02,ML03a}, G.~Mittag--Leffler introduced the now classical function $ E_\alpha(z) $, an entire function on the complex plane that deforms the exponential via the absolutely convergent series
\[
E_\alpha(z) := \sum_{n=0}^\infty \frac{z^n}{\Gamma(\alpha n + 1)}, \quad z, \alpha \in \mathbb{C}.
\]
In the following decades, $ E_\alpha $ and its generalizations—with multiple parameters $ \bm{\alpha} $ and several variables $ \bm{z} $—were extensively studied \cite{ML03b,ML04,ML05,Wim05a,Wim05b,Mal05,Bar06,Mel10,Wri35,Fox28,Mei36,Mei46,Hum53,HA53,Dzh54,Pra71,KS96}, finding applications across a broad spectrum of pure and applied mathematics. For a comprehensive account, we refer to the monograph \cite{GKMR20}.

In this paper, we introduce a class of multivariable special functions of Mittag--Leffler type, defined in terms of a fixed integer matrix $ \mathsf{M} = (m_i^a) \in \mathbb{Z}^{r \times N} $ satisfying the positivity condition
\beq\label{eqintro2}
\sum_{i=1}^N m^a_i > 0, \quad \text{for all } a = 1, \dots, r.
\eeq
Given such a matrix, we define the \emph{generalized Mittag--Leffler function} $ \mathcal{E}_{\mathsf{M}}(\bm{s}, z) \in \eu{O}(\mathbb{C}^r \times \widetilde{\mathbb{C}^*}) $ by
\[
\mathcal{E}_{\mathsf{M}}(\bm{s}, z) := \sum_{\bm{d} \in \mathbb{N}^r} \frac{z^{\sum_{a=1}^r \sum_{i=1}^N m^a_i (d_a + s_a)}}{\Gamma_{\mathsf{M}}(\bm{s}, \bm{d})}, \quad \Gamma_{\mathsf{M}}(\bm{s}, \bm{d}) := \prod_{i=1}^N \Gamma\left(1 + \sum_{a=1}^r m^a_i (d_a + s_a)\right).
\]
After proving the absolute convergence of this series, we establish a Mellin--Barnes-type integral representation:
\[
\mathcal{E}_{\mathsf{M}}(\bm{s}, z) = \frac{1}{(2\pi i)^r} \int_{\mathfrak{H}} \Upsilon_{\mathsf{M}}(\bm{\zeta}, \bm{s}, z) \, \mathrm{d}\zeta_1 \dots \mathrm{d}\zeta_r,
\]
where $ \mathfrak{H} $ is a suitable multidimensional Hankel contour and $ \Upsilon_{\mathsf{M}} $ is an explicit integrand involving ratios of Gamma functions and exponential terms (see Theorem~\ref{thMBML}). Associated to $ \mathcal{E}_{\mathsf{M}}(\bm{s}, z) $, we also define the univariate functions $ \mathcal{E}_{\mathsf{M}, \bm{\alpha}}(z) $ via
\[
\mathcal{E}_{\mathsf{M}, \bm{\alpha}}(z) := \left. \frac{\partial^{|\bm{\alpha}|} \mathcal{E}_{\mathsf{M}}}{\partial s_1^{\alpha_1} \dots \partial s_r^{\alpha_r}} \right|_{\bm{s} = 0}, \quad \bm{\alpha} \in \mathbb{N}^r,\quad |\bm{\alpha}| := \sum_{i=1}^r \alpha_i.
\]
\begin{rem}
All the functions introduced so far either belong to, or can be reduced to (via change of variables and/or specialization along diagonals), the broad class of multivariate generalized hypergeometric functions as defined, for instance, in \cite{SD72,Ext76}. More generally, they may be viewed as particular instances of Horn-type series or $\Gamma$-series in the sense of I.M.\,Gel'fand and collaborators, see \cite{GGM92} and references therein. Notably, there has been a recent surge of interest in studying the differential properties of such special functions with respect to their parameters, motivated by a wide range of applications across Mathematics and Physics. See, for example, \cite{Bry16,Dun17,Ape20,AM20,BK20,GKMR20,RGM25}.
\end{rem}
The simplest case, corresponding to $\mathsf{M}=(1)$, was introduced in our previous work \cite{Cot24}. It yields the bivariate function
\[
\mathcal{E}(s, z) := \mathcal{E}_{(1)}(s, z) = \sum_{n=0}^\infty \frac{z^{n + s}}{\Gamma(1 + n + s)},
\]
which is related to the B\"ohmer--Tricomi incomplete Gamma function $ \gamma^*(s, z)\in\eu O(\C^2)$ via
\[
\mathcal{E}(s, z) = e^z z^s \gamma^*(s, z), \quad (s, z) \in \mathbb{C} \times \widetilde{\mathbb{C}^*},
\]from which several integral representations follow \cite{B\"oh39,Tri50a,Tri50b,Tri54}.
The corresponding functions $ \mathcal{E}_{(1), (k)}(z) $, denoted simply by $ \mathcal{E}_k(z) $ in \cite{Cot24}, are obtained as derivatives in $ s $ at zero.

The second analytical tool consists of two families of integral transforms.  

Given a $(s+1)\times h$-matrix with complex entries
\beq\label{introMmatrix} \eu M=
\begin{pmatrix}
\al_1&\dots&\al_h\\
\bt^1_1&\dots&\bt^1_h\\
\vdots&\ddots&\vdots\\
\bt^s_1&\dots&\bt^s_h
\end{pmatrix},\quad \al_j\neq 0,\quad \sum_{i=1}^s\bt^i_j\neq 0,\quad \text{for all }j=1,\dots,h,
\eeq we define the Laplace and Borel $\eu M$-multitrasforms of functions $(\Phi_j)_{j=1}^h\in\eu O(\Tilde{\C^*})^{\times h}$, respectively by:
\[
\mathscr L_{\eu M}\left[\Phi_1,\dots,\Phi_h\right](z):=\int_{(\R_{>0})^{\times s}}\left[\prod_{j=1}^h\Phi_j\left(z^{\al_j\sum_{i=1}^s\bt^i_j}\prod_{i=1}^s\zeta_i^{\bt^i_j}\right)\right]\exp\left(-\sum_{i=1}^s\zeta_i\right){\rm d}\zeta_1\dots{\rm d}\zeta_s,
\]
\[
\mathscr B_{\eu M}[\Phi_1,\dots,\Phi_h](z):=
\frac{1}{(2\pi\sqrt{-1})^s}\int_{\frak H}\left[\prod_{j=1}^h\Phi_j\left(z^{\frac{1}{\al_j\sum_{i=1}^s\bt^i_j}}\prod_{i=1}^s\zeta_i^{-\bt^i_j}\right)\right]\frac{e^{\sum_{i=1}^s\zeta_i}}{\zeta_1\dots\zeta_s}{\rm d}\zeta_1\dots {\rm d}\zeta_s,
\]whenever the integrals converge. The contour $\frak H$ is a product of contours of Hankel type in each $\zeta_i$-plane, with $i=1,\dots, s$, originating at $-\infty - \sqrt{-1}\varepsilon$, encircling the origin in the positive direction, and ending at $-\infty + \sqrt{-1}\varepsilon$, for some small $\varepsilon > 0$.

We are now ready to summarize the main results of this paper.

The first main results, Theorem \ref{MTH0}, concerns the space of master functions of Fano toric smooth projective varieties. Any such variety  $X$ can be realized as a GIT quotient $X=\C^N/\!\!/_{\om} {\sf T}^r$, of an open subset of some $\C^N$ by by a torus action determined by a {\it weight matrix} ${\sf M}\in\Z^{r\times N}$ and a choice of {\it stability condition} $\om\in\Z^N$ lying in the ample cone. The Fano condition assures that the matrix $\sf M$ can be taken to satisfy the positivity conditions \eqref{eqintro2}. See Sections \ref{sectoric} and \ref{secmstoric} for more details. 

Theorem \ref{MTH0} claims the existence of linear combination with integer coefficients of the generalized Mittag--Leffler functions $\mc E_{{\sf M},\bm\al}(z)$, with $|\bm\al|=0,\dots,\dim_\C X$, forming a basis of the space of master functions of $X$ at $0\in M_X$. In particular, the function $\mc E_{{\sf M},\bm 0}(z)$ is identified with a generating function of Gromov--Witten invariants with {\it gravitational descendants} known as {\it quantum period} $X$ \cite{CCGK16}. Consequently, explicit Mellin--Barnes integral representations for master functions, and hence solutions of the qDE, are obtained.

The second and third main results can be formulated starting within a common framework. Consider a product of Fano smooth projective varieties $X=X_1\times \dots\times X_h$, and let $E$ be a $s$-rank split vector bundle on $X$, given by the direct sum of tensor products of fractional powers of each $\det TX_j$, with $j=1,\dots, h$, that is
\[E=\bigoplus_{i=1}^s \bigboxtimes_{j=1}^{h} L_j^{\otimes d_{ij}},\quad d_{ij}\in\N^*,\qquad \det TX_j=L_j^{\otimes \ell_j},\quad \ell_j\in\N^*,\quad j=1,\dots,h,
\]where $L_j\in {\rm Pic}(X_j)$, $j=1,\dots,h$, is an ample line bundle. Under the assumption
\[\sum_{i=1}^sd_{ij}<\ell_j,\quad j=1,\dots, h,
\]both the zero locus $Y\subseteq X$ of a regular section of $E$, and the total space of the projective bundle $P=\Pb(\eu O_x\oplus E^*)$ are smooth Fano varieties.

Any class $\bm{\delta} \in H^2(X,\C)$ induces classes $\iota^* \bm{\delta} \in H^2(Y,\C)$ and $\pi^* \bm{\delta} \in H^2(P,\C)$ via pullback, where $\iota \colon Y \to X$ and $\pi \colon P \to X$ denote the natural maps. By the Künneth isomorphism, $\bm{\delta}$ decomposes as a sum $\sum \delta_j$, with $\delta_j \in H^2(X_j, \C)$.

The second main results, Theorem \ref{TH1}, claims that each master function of $Y$ at $\iota^*\bm\dl$ can be reconstructed from the master functions of $X_j$ at $\dl_j$ via a Laplace $\eu M$-multitransform with $\eu M\in\Q^{(s+1)\times h}$, up to an exponential factor. Namely, there exits a rational number $c_{\bm\dl}\in\Q$ such that each master function of $Y$ at $\iota^*\bm\dl$ is of the form
\[e^{c_{\bm\dl}z}\mathscr L_{\eu M}[\Phi_1,\dots,\Phi_j],
\]where $\Phi_j$ is a master function of $X_j$ at $\dl_j$, and the entries matrix $\eu M$ in \eqref{introMmatrix} are given by
\[\al_j=\frac{\ell_j-\sum_{i=1}^sd_{ij}}{\sum_{i=1}^sd_{ij}},\qquad \bt^i_j=\frac{d_{ij}}{\ell_j},\qquad i=1,\dots,s,\quad j=1,\dots,h.
\]

The third main result, Theorem \ref{mt2}, asserts that each master function of $P$ at $\pi^* \bm{\delta}$ arises via a Borel $\eu M$-multitransform, with $\eu M\in\Q^{(s+1)\times (h+1)}$,  of the master functions of $X_j$ and the Mittag--Leffler functions $\mc E_k = \mc E_{(1),(k)}$, with $k = 0, \dots, \dim_\C X$. More precisely, each master function of $P$ at $\pi^*\bm\dl$ is of the form
\[\mathscr B_{\eu M}[\Phi_1,\dots,\Phi_h,\mc E_k],\quad k=0,\dots,\dim_\C X,
\]where $\Phi_j$ is a master function of $X_j$ at $\dl_j$, and $\eu M$ as in \eqref{introMmatrix} has entries
\[
\al_j = \frac{\ell_j^2}{(\sum_{i=1}^s d_{ij})[(\sum_{i=1}^s d_{ij}) - \ell_j]}, \quad \al_{h+1} = \frac{1}{s(s+1)}, \qquad \bt^i_j = -\frac{d_{ij}}{\ell_j}, \quad \bt^i_{h+1} = 1,
\]for $i=1,\dots,s$ and $j=1,\dots,h$.
\vskip2mm
\noindent 1.5.\,\,The paper is organized as follows. Section~\ref{sec1} reviews foundational material in Gromov--Witten theory and quantum cohomology, with a focus on the associated Frobenius manifold structures in both formal and analytic settings. In Section~\ref{sec2}, we introduce quantum differential equations, the notions of cyclic stratum and $\mc A_\La$-stratum, as well as the definition of master functions.

Section~\ref{sec3} is devoted to the study of generalized Mittag--Leffler functions: we establish their fundamental properties and integral representations. After a brief overview of the GIT construction of toric varieties, we relate the space of master functions on Fano toric varieties to generalized Mittag--Leffler functions, and formulate the first main result.

In Section~\ref{sec4}, we introduce analytic Borel--Laplace $\eu M$-multitransforms and present the second and third main results. Section~\ref{sec5} contains detailed proofs of all the main theorems. Finally, Section~\ref{sec6} discusses examples and applications, including Fano toric complete intersections, Fano bundles over toric varieties, and strict del Pezzo surfaces.
\vskip2mm
\noindent{\bf Acknowledgements.} We are grateful to C.\,Hertling, H.\,Iritani, P.\,Jossen, C.\,Sabbah, A. Smirnov, and A.\,Varchenko,  for several useful discussions. This research was supported by the Fundação para a Ciência e a Tecnologia (FCT), under project UIDB/00208/2020.

\section{Gromov--Witten theory and quantum cohomology}\label{sec1}

\subsection{Notations}
Let $X$ be a smooth complex projective variety. Introduce the $\C$-vector space $\eu{H}_X:=\bigoplus_{k\geq 0}H^{2k}(X,\C)$, together with a fixed homogeneous basis $\bm T:=(T_0,\dots, T_n)$, with dual coordinates $\bm t=(t^0,\dots, t^n)$. In what follows, without loss of generality we assume that $T_0=1\in H^0(X,\C)$, and that $T_1,\dots, T_r$ span the subspace $H^2(X,\C)$. 
\vskip1,5mm
The Poincar\'e pairing $\bm\eta$ is the symmetric non-degenerate bilinear form $\bm\eta\colon \bigodot^2\eu H_X\to\C$ defined by
\[\bm\eta(\al_1\odot\al_2):=\int_X\al_1\cup\al_2,\quad \al_1,\al_2\in\eu H_X,
\]where $\cup$ denotes the cup product. In the basis $\bm T$, the Poincar\'e pairing is represented by the Gram matrix $\eta=(\eta_{\al\bt})_{\al,\bt=0}^n$, with $\eta_{\al\bt}:=\bm\eta(T_\al\odot T_\bt)$.

\subsection{Gromov--Witten invariants} Set $H_2(X,\Z)_{\rm tf}:=H_2(X,\Z)/{\rm torsion}$, the free part of the homology group.
\vskip1,5mm
Given $\bt\in H_2(X,\Z)_{\rm tf}$, consider the set of triples $(C,\bm x,f)$ where \begin{enumerate}
\item $C$ is a genus 0 algebraic curve with (at most) nodal singularities,
\item $\bm x=(x_1,\dots, x_k)$ is a collection of pairwise distinct points of the smooth locus of $C$,
\item $f\colon C\to X$ is a morphism such that $f_*[C]=\bt$.
\end{enumerate}
A morphism $(C,\bm x,f)\to (C',\bm x', f')$ is a morphism $\phi\colon C\to C'$ such that $f=f'\circ\phi$, and $\phi(x_i)=x_i'$ for $i=1,\dots,k$.
\vskip1,5mm
Denote by $X_{0,k,\bt}$ the Kontsevich--Manin moduli stack of stable $k$-pointed rational maps of degree $\bt$ with target $X$, parametrizing isomorphism classes of triples $(C,\bm x,f)$ with finite automorphism groups. It has virtual dimension  
\[
\on{vir\,dim}_\C X_{0,k,\bt}=\dim_\C X-3+k+\int_\bt c_1(X).
\]  
The stack $X_{0,k,\bt}$ is non-empty only if $\bt$ belongs to the {\it Mori cone} ${\rm NE}(X)$, i.e., the semigroup of classes in $H_2(X,\Z)_{\rm tf}$ that can be represented by algebraic curves.  
The {\it K\"ahler cone} $\mc K_X \subset H^2(X,\R)\cap H^{1,1}(X,\C)$ consists of the cohomology classes of K\"ahler forms on $X$. The Mori cone ${\rm NE}(X)$ is dual to $\mc K_X$ in the following way. Let ${\rm NE}(X)_\R$ be the smallest convex cone in $H_2(X,\R)$ containing ${\rm NE}(X)$. The dual of its closure coincides with $\overline{\mc K_X}$. This follows from the identification of $\mc K_X$ with the ample cone, together with Kleiman’s ampleness criterion. In particular, for any $\bt \in {\overline{\rm NE}(X)_\R} \setminus \{0\}$ and any $\om \in \mc K_X$, we have  
\beq\label{ample}
\int_\bt\om>0.
\eeq See \cite{KM98,Voi02,Huy04,Laz04}. 

\begin{rem}
The dual cone of $\overline{\rm NE}(X)_\R$ is commonly referred to as the {\it nef cone} (or {\it cone of nef divisors}), usually denoted by ${\rm Nef}(X)$. Thus, we obtain the relation ${\rm Nef}(X) = \overline{\mc K_X}$.
\end{rem}
The construction of \cite{BF97} allows to develop a well-behaved intersection theory on $X_{0,k,\bt}$, defining a virtual fundamental cycle $[X_{0,k,\bt}]^{\rm vir}\in CH_D(X_{0,k,\bt})$ of dimension $D=\on{vir\,dim}_\C X_{0,k,\bt}$. 
Given cohomological classes $a_i\in \eu{H}_X$ and numbers $d_i\in\N$, for $i=1,\dots,k$, the {\it genus 0 descendant Gromov--Witten invariant} $\bra\tau_{d_1}a_1,\dots\tau_{d_k}a_k\ket_{0,k,\bt}^X$ is defined as the integral
\beq\label{gw1}
\bra\tau_{d_1}a_1,\dots\tau_{d_k}a_k\ket_{0,k,\bt}^X=\int_{[X_{0,k,\bt}]^{\rm vir}}\prod_{j=1}^k\on{ev}_j^*a_j\cup\psi_j^{d_j}\quad \in\Q,
\eeq where, for all $i=1,\dots,k$, $\on{ev}_i\colon X_{0,k,\bt}\to X$ is the evaluation morphism $(C,\bm x,f)\mapsto f(x_i)$, and where $\psi_i:=c_1(\mathscr L_i)$, with $\mathscr L_i\to X_{0,k,\bt}$ the $i$-th tautological line bundle whose fiber at $(C,\bm x, f)$ is given by $\mathscr L_i|{(C,\bm x, f)}=T^*_{x_i}C$.

\begin{rem}\label{gwzero}
In the case $d_1=\dots=d_k=0$, the invariant $\bra \tau_0a_1,\dots\tau_0 a_k\ket_{0,k,\bt}^X$ is simply denoted by $\bra a_1,\dots, a_k\ket_{0,k,\bt}^X$ and called {\it primary} Gromov--Witten invariant. If $\bt \not\in {\rm NE}(X)$ or if the integrand in \eqref{gw1} has a zero component in $H^{2D}(X_{0,k,\bt})$, with $D=\on{vir\,dim}_\C X_{0,k,\bt}$, then the Gromov--Witten invariant $\langle\tau_{d_1}\gm_1,\dots, \tau_{d_k}\gm_k\rangle^X_{0,k,\bt}$ is set to zero.
\end{rem}

\subsection{Gromov--Witten potential} Fix a K\"ahler form $\om$ on $X$. The {\it Novikov ring} $\La_{X,\om}$ is defined as the ring of formal power series  
\[
\sum_{\bt\in H_2(X,\Z)_{\rm tf}}a_\bt{\bf Q^\bt}, \quad a_\bt\in\C,
\]
where for any $C\in\R$, only finitely many terms satisfy $\int_\bt\om<C$. This can be seen as a completion of the group ring $\C[H_2(X,\Z)_{\rm tf}]$, allowing infinite sums in the direction of increasing $\int_\bt\om$.  

Following E.\,Witten \cite{Wit93}, we introduce the {\it big phase space} $\eu P_X := \eu H_X^\N$, an infinite product of copies of $\eu H_X$. The subspace given by the first factor $\eu H_X$, is called the {\it small phase space}. Using the fixed basis $\bm T=(T_0,\dots,T_n)$, we denote by $(\tau_kT_0,\dots,\tau_kT_n)$ the basis of the $k$-th copy of $\eu H_X$. Any point $\bm\gm\in \eu P_X$ is then written as  
\[
\bm\gm=\sum_{\al=0}^n\sum_{k=0}^\infty t^{\al,k}\tau_kT_\al,
\]
where the coordinates $\bm t^\bullet=(t^{\al,k})_{\al,k}$ serve as formal parameters.  

The {\it genus 0 total descendant potential} $\eu F^X_0$ is the generating function for descendant Gromov--Witten invariants:  
\[
\eu F^X_0(\bm t^\bullet, {\bf Q}):=\sum_{k=0}^\infty\sum_{\bt}\sum_{\al_1,\dots,\al_k=0}^n\sum_{p_1,\dots,p_k=0}^\infty\frac{t^{\al_1,p_1}\dots t^{\al_k,p_k}}{k!}\langle\tau_{p_1}T_{\al_1}\dots\tau_{p_k}T_{\al_k}\rangle^X_{0,k,\bt}{\bf Q}^\bt.
\]

When restricted to the small phase space, setting $t^{\al,0}=t^\al$ and $t^{\al,p}=0$ for $p>0$, this simplifies to the {\it genus 0 Gromov--Witten potential}:  
\[
F^X_0(\bm t,{\bf Q}):=\sum_{k=0}^\infty\sum_\bt\sum_{\al_1,\dots, \al_k=0}^n\frac{t^{\al_1}\dots t^{\al_k}}{k!}\langle T_{\al_1},\dots, T_{\al_k}\rangle^X_{0,k,\bt}{\bf Q}^\bt.
\]

\begin{rem}
The series $\eu F^X_0(\bm t^\bullet, {\bf Q})$ and $F^X_0(\bm t,{\bf Q})$ are well-defined elements of $\La_{X,\om}[\![\bm t^\bullet]\!]$ and $\La_{X,\om}[\![\bm t]\!]$, respectively. This follows by Remark \ref{gwzero} and the ampleness condition \eqref{ample}.
\end{rem}

\subsection{Quantum cohomology as Frobenius manifold}
Quantum cohomology provides a rich geometric structure, offering a natural example of a Frobenius manifold. 
The construction progresses from a formal setting over the Novikov ring to an analytic structure under appropriate convergence assumptions.
\vskip2mm
{\bf Formal Frobenius manifold over the Novikov ring. }
We consider the finite rank free $\La_{X,\om}$-module
\[
H := \eu H_X \otimes_\C \La_{X,\om},
\]
with dual module $H^T$, and define the algebra $ K = \La_{X,\om}[\![H^T]\!] $, identified with the ring of formal power series in coordinates $ \bm t = (t^0, \dots, t^n) $. The formal spectrum $ M_X^{\rm for} = \mathrm{Spf}(K) $ is then endowed with the structure sheaf $ \eu{O}_M \cong K $ and a space of formal vector fields $ \eu{T}_M \cong H_K $, where $ H_K = H \otimes_{\La_{X,\om}} K $.
The Poincaré pairing extends to a non-degenerate symmetric bilinear form $ \bm\eta: \eu{T}_{M_X^{\rm for}} \otimes \eu{T}_M \to \eu{O}_{M_X^{\rm for}} $. 
\vskip1,5mm
The Gromov--Witten potential $ F^X_0 $ defines a formal function on $ M_X^{\rm for} $, satisfying two remarkable properties \cite{KM94}:

\begin{enumerate}
    \item[(FM1)] \textbf{Quasi-homogeneity}: There exists a distinguished vector field, called {\it Euler vector field}
    \[
    \mathsf E = c_1(X) + \sum_{\alpha=0}^n \left(1 - \frac{1}{2} \deg T_\alpha \right) t^\alpha \frac{\partial}{\partial t^\alpha}
    \]
    such that
    \[
    \mathsf E F^X_0 = (3 - \dim_\C X) F^X_0 + Q,
    \]
    where $ Q $ is a quadratic polynomial in $ \bm t $.
    \item[(FM2)] \textbf{WDVV Equations}: The function $ F^X_0 $ satisfies the Witten--Dijkgraaf--Verlinde--Verlinde (WDVV) equations, 
    \beq\label{wdvv}
\sum_{\la,\nu}\frac{\der^3 F^X_0}{\der t^\al\der t^\bt\der t^\la}\eta^{\la\nu}\frac{\der^3 F^X_0}{\der t^\nu\der t^\gm\der t^\dl}=
\sum_{\la,\nu}\frac{\der^3 F^X_0}{\der t^\dl\der t^\bt\der t^\la}\eta^{\la\nu}\frac{\der^3 F^X_0}{\der t^\nu\der t^\gm\der t^\al}.
\eeq
\end{enumerate}

The $K$-bilinear product $\circ$ of formal vector fields on $M_X^{\rm for}$ defined by
\[
T_\alpha \circ T_\beta = \sum_{\gamma} c^\gamma_{\alpha\beta} T_\gamma, \quad c^\gamma_{\alpha\beta} = \sum_{\lambda} \eta^{\gamma\lambda} \frac{\partial^3 F^X_0}{\partial t^\alpha \partial t^\beta \partial t^\lambda},\quad \al,\bt,\gm=0,\dots,n,
\]
is commutative, associative (due to the WDVV equations), and compatible with $ \bm\eta $, that is
\[\bm \eta(v_1\circ v_2, v_3)=\bm\eta(v_1,v_2\circ v_3),\quad v_1,v_2,v_3\in H_K.
\] The element $ e = T_0 $ serves as the identity, making $ (H_K, \circ, \eta, e) $ a Frobenius algebra. 
\vskip1,5mm
The structure $(M_X^{\rm for}, F^X_0, \eta, e, \mathsf E)$ defines a \textit{formal Frobenius manifold} over $ \La_{X,\om} $, known as the \textit{even quantum cohomology of $ X $}.
\vskip1,5mm
While one could consider the full cohomology $ H^\bullet(X, \C) $, yielding a Frobenius super-manifold structure, we focus on the even part $\eu H_X$ only. It turns out this is not a big restriction: under generic algebraic conditions on $ (H_K, \circ) $, the even part turns out to be the only relevant sector, as shown by the following result.
\begin{thm}\cite{HMT09}
If $ (H_K,\circ) $ is semisimple, then $ X $ has no odd cohomology and is of Hodge--Tate type, i.e. $ h^{p,q}(X) = 0 $ for $ p \neq q $. \qed
\end{thm}

{\bf Formal Frobenius Manifold over $ \C $.}
By fixing a basis $(\bt_1,\dots,\bt_r)$ of $H_{2}(X,\Z)_{\rm tf}$, we identify the Novikov ring $ \La_{X,\om} $ with a ring of formal power series in $ r $ variables $\bm Q=( Q_1, \dots, Q_r) $, where  $Q_i:={\bf Q}^{\bt_i}$ for all $i=1,\dots,r$. In what follows we will use the multi-index notation $\bm Q^\bt=\prod_{j=1}^rQ_j^{n_j}$ with $\bt=\sum_{j=1}^rn_j\bt_j$.
\vskip1,5mm
\begin{prop}\label{assconv1}\cite{Cot24} The following conditions are equivalent.
\begin{enumerate}
\item There exists $\bm q_o\in(\C^*)^r$ such that the series $\sum_\bt\bra T_{\al_1},\dots,T_{\al_k}\ket_{0,k,\bt}^X \bm q_o^\bt$ is convergent for all $\al_1,\dots,\al_k\in\{0,\dots,n\}$ and $k\in\N$.
\item There exists $\bm q_o\in(\C^*)^r$ such that the series $\sum_\bt\bra v_1,\dots,v_k\ket_{0,k,\bt}^X \bm q_o^\bt$ is convergent for all $v_1,\dots,v_k\in\eu H_X$ and $k\in\N$.
\item For all $\bm q\in(\C^*)^r$, the series $\sum_\bt\bra v_1,\dots,v_k\ket_{0,k,\bt}^X \bm q^\bt$ is convergent for all $v_1,\dots,v_k\in\eu H_X$ and $k\in\N$.\qed
\end{enumerate}
\end{prop}
\vskip1,5mm
\noindent{\bf Assumption A.} The equivalent conditions of Proposition \ref{assconv1} hold.
\vskip1,5mm
\begin{prop}\cite{Cot24}
If $X$ is Fano, Assumption A holds.\qed
\end{prop}
Under Assumption A, for any $\bm q\in (\C^*)^r$, we obtain a \textit{formal Frobenius manifold over $ \C $},
\[
\left(\mathrm{Spf}\, \C[\![t^0,\dots,t^n]\!],\quad F^X_0|_{{\bm Q}=\bm q},\quad \eta,\quad e,\quad \mathsf E\right),
\]
which we call the \textit{specialized even quantum cohomology of $ X $ at $ \bm q $}.

\vskip2mm
{\bf Dubrovin--Frobenius manifold structure.}
To extend quantum cohomology beyond the formal setting, we consider the case where the power series $ F^X_0|_{{\bm Q}=\bm q} $ is convergent in a neighborhood of the origin.

\begin{prop}\label{assconv2}\cite{Cot24}
The following conditions are equivalent.
\begin{enumerate}
\item There exists $\bm q_o\in(\C^*)^r$ such that $F^X_0|_{\bm Q=\bm q_o}$ is a convergent series in $\bm t$.
\item For all $\bm q\in(\C^*)^r$, $F^X_0|_{\bm Q=\bm q}$ is a convergent series in $\bm t$.\qed
\end{enumerate}
\end{prop}

\noindent{\bf Assumption B.} The equivalent conditions of Proposition \ref{assconv2} hold. 

\begin{thm}\cite{Cot21,Cot24a,Cot24}
Let Assumption A hold. If the even quantum cohomology of $X$ specialized at some $\bm q_o\in(\C^*)^r$ is semisimple, then Assumption B holds.\qed
\end{thm}

Under Assumption B, the series $ F^X_0|_{{\bm Q}=\bm q} $ share the same convergence domain $ M_X \subseteq \eu H_X $ for all $ \bm q \in (\C^*)^r $. Without loss of generality, we specialize $ F^X_0 $ at $ {\bm Q} = \bm 1 = (1,1,\dots,1) $.  

Denote by $\eu O_{M_X}$ the sheaf of holomorphic functions on $M_X$. Let $ TM_X $ and $ T^*M_X $ be the holomorphic tangent and cotangent bundles of $ M_X $, respectively. At each $ p \in M_X $, there is a canonical identification $ T_pM_X \cong \eu H_X $ via $ \frac{\partial}{\partial t^\alpha} \mapsto T_\alpha $. Through this, the Poincar\'e metric defines a holomorphic, symmetric, non-degenerate $ \eu{O}_{M_X} $-bilinear 2-form $ \bm\eta\in\Gm(\bigodot^2T^*M_X) $ with a flat Levi-Civita connection $ \nabla $.  
The holomorphic $(1,2)$-tensor  $\bm c\in\Gm(TM_X\otimes \bigodot^2T^*M_X)$, with components
\[
c^\alpha_{\beta\gamma} := \sum_\lambda \eta^{\alpha\lambda} \frac{\partial^3}{\partial t^\lambda \partial t^\beta \partial t^\gamma} F^X_0|_{{\bm Q}=\bm 1},
\]  
defines a holomorphic Frobenius algebra structure on each tangent space $ T_pM_X $, varying holomorphically with $ p $. The constant vector field $ e = T_0 $ serves as the {\it unit vector field}, while the {\it Euler vector field} $ \mathsf E $ satisfies  
\[
\mathfrak{L}_{\mathsf E} \bm c = \bm c, \quad \mathfrak{L}_{\mathsf E} \bm \eta = (2 - \dim_\C X) \bm \eta.
\]  
This structure $ Q\!H^{\rm ev}_X:=(M_X, \bm \eta, \bm c, e, \mathsf E) $ defines an analytic Frobenius manifold, aka {\it Dubrovin--Frobenius manifold} \cite{Dub96,Dub99}. In this paper, we refer to this Dubrovin–Frobenius manifold\footnote{By a slight abuse of notation, we will also write $ p \in QH_X^{\mathrm{ev}} $ to mean that $ p \in M_X $, where $ M_X $ is the underlying manifold equipped with the Dubrovin--Frobenius manifold structure.
} as the even quantum cohomology of $ X $.  

\section{Quantum differential equations and master functions}\label{sec2}
\subsection{Extended deformed connection}
Let $ Q\!H^{\rm ev}_X:=(M_X, \bm \eta, \bm c, e, \mathsf E) $ be the Dubrovin--Frobenius manifold from the previous section. As before, $TM_X$ (resp.\,\,$T^*M_X$) is its holomorphic tangent (resp.\,\,cotangent) bundle, and $\eu T_{M_X}$ denotes its sheaf of holomorphic sections. 
\begin{defn}
The \emph{grading operator} $\bm \mu\in \Gm({\rm End}\,TM_X)$ and the operator $\bm {\mc U}\in \Gm({\rm End}\,TM_X)$ are given by
\beq
\bm\mu(v):=\frac{2-\dim_\mathbb CX}{2}v-\nabla_v\mathsf E, \quad
\bm{\mc U}(v):=\mathsf E\circ v,\quad v\in\Gamma(TM_X).
\eeq
Denote by $\mu$ and $\mc U$ their matrices in $\nabla$-flat coordinates $\bm t$.
\end{defn}
\begin{lem}\label{lemsymUmu}
For $v_1,v_2\in\Gm(TM_X)$,
\[\bm\eta({\bm{\mc U}}(v_1),v_2)=\bm\eta(v_1,{\bm {\mc U}}(v_2)),\quad 
\bm\eta({\bm{\mu}}(v_1),v_2)=-\bm\eta(v_1,{\bm {\mu}}(v_2)).
\]
\end{lem}

Consider $\widehat M_X:=\C^*\times M_X$ with projection $\pi\colon\widehat M_X\to M_X$. The pull-back bundle $\pi^*TM_X$ carries the lifted tensors $\bm\eta,\bm c, \mathsf E, \bm\mu,\bm{\mathcal U}$ and a uniquely lifted Levi--Civita connection $\nabla$ satisfying
\[\nabla_\frac{\partial}{\partial z}v=0\quad \forall v\in\pi^{-1}\eu T_{M_X},
\]where $\pi^{-1}\eu T_{M_X}$ is the sheaf of holomorphic sections of $\pi^*TM_X$ constant on the fibers of $\pi$.
\begin{defn}
The \emph{extended deformed connection} $\widehat \nabla$ on $\pi^*TM_X$ is defined as
\begin{align}
\widehat\nabla_{w}v&=\nabla_wv+z\cdot w\circ v,\\
\widehat\nabla_{\frac{\partial}{\partial z}}v&=\nabla_{\partial_z}v+ \bm{\mc U}(v)-\frac{1}{z}\bm \mu(v),
\end{align}
for $v,w\in \Gamma(\pi^*TM_X)$.
\end{defn}
\begin{thm}\cite{Dub96,Dub99}
The connection $\widehat\nabla$ is flat.\qed
\end{thm}

\subsection{The quantum differential equation}
The connection $\widehat\nabla$ induces a flat connection on $\pi^*T^*M_X$. A flat section $\varpi\in\Gamma(\pi^*T^*M_X)$ corresponds via $\bm\eta$ to a vector field $\sig\in\Gamma(\pi^*TM_X)$ satisfying \footnote{\,\,The joint system \eqref{eq1}, \eqref{qde} is written in matrix notations: the vector field $\sig$ is represented as a column vector with components $\sig^\al(z,\bm t)$ wrt $\frac{\der}{\der t^\al}$, with $\al=0,\dots,n$.}
\begin{align}
\label{eq1}
\frac{\partial}{\partial t^\alpha}\sig&=z\,\mathcal C_\alpha\sig,\quad \al=0,\dots,n,\\
\label{qde}
\frac{\partial}{\partial z}\sig&=\left(\mathcal U+\frac{1}{z}\mu\right)\sig.
\end{align}
Here $\mc C_\al$ is the $(1,1)$-tensor with components $\left(\mc C_\al\right)^\bt_\gm=c^\bt_{\al\gm}$, for $\al,\bt,\gm=0,\dots,n$.
\begin{defn}
The \emph{quantum differential equation} ({\it qDE}) of $X$ is equation \eqref{qde}.
\end{defn}

Using Lemma \ref{lemsymUmu}, the system can be rewritten for the covector field $\varpi=\eta\sig$ as
\begin{align}
\label{eq1.2}
\frac{\der}{\der t^\al}\varpi&=z\,\mc C_\al^T\varpi,\\
\label{qde.2}
\frac{\der}{\der z}\varpi&=\left(\mc U^T-\frac{1}{z}\mu\right)\varpi.
\end{align}
Here $\varpi=\eta\sig$ is a column vector with components $(\varpi_\al)_{\al=0}^n$.

\subsection{Cyclic stratum, $\mc A_\La$-stratum, master functions}\label{scyc}
For $k\in\N$, define the vector fields $e_k\in\Gm(\pi^*TM_X)$ by
\beq\label{eqvecek}
e_k:=\edc_{\frac{\der}{\der z}}^ke.
\eeq
\begin{defn}
The \emph{cyclic stratum} $\widehat M_X^{\rm cyc}$ is the maximal open set $U\subseteq\Hat M_X$ where $(e_k)_{k=0}^{n}$ defines a global frame for the trivial bundle $\pi^*TM_X|_U$. The dual coframe $(\omega_j)_{j=0}^{n}$ is defined by $\langle\omega_j,e_k\rangle=\delta_{jk}$. The frame $(e_k)_{k=0}^{n}$ is called {\it cyclic frame}, and its dual $(\omega_j)_{j=0}^{n}$ {\it cyclic coframe}.
\end{defn}

\begin{defn}
The $\La$-{\it matrix} is the matrix-valued function $\La=(\La_{i\al}(z,p))$, holomorphic on $\widehat M_X^{\rm cyc}$, defined by
\beq\label{Lam}
\frac{\der}{\der t^\al}=\sum_{i=0}^{n}\La_{i\al}e_i,\quad \al=0,\dots,n.
\eeq
\end{defn}

\begin{thm}\label{strLa}\cite[Th.\,2.20]{Cot22} The function
$\det\La$ is meromorphic on $\Pb^1\times M_X$ with form
\[\det\La(z,p)=\frac{z^{\binom{n}{2}}}{z^{\binom{n}{2}}A_0(p)+\dots+A_{\binom{n}{2}}(p)},
\]
where $A_h(p)$, with $h=0,\dots,\binom{n}{2}$ are holomorphic. If $n>1$ and the eigenvalues of $\bm\mu$ are not distinct, then $A_{\binom{n}{2}}=0$.\qed
\end{thm}

Define $\bar n:=\min\{j\in\N\colon A_{h}(p)=0\ \forall p\in M_X,\,\forall h>j\}$. Then
\[\det \La=\frac{z^{\bar n}}{z^{\bar n} A_0(p)+z^{\bar n-1}A_1(p)\dots+A_{\bar n}(p)}.
\]
\begin{defn}
The set $\mc A_{\La}\subseteq M_X$ is defined as
\[\mc A_\La:=\left\{p\in M_X\colon\quad A_0(p)=\dots=A_{\bar n}(p)=0\right\}.
\]
\end{defn}

\begin{thm} \cite[Th.\,2.29]{Cot22}\label{thsqde}
For $p\in M_X\setminus\mc A_{\La}$, equation \eqref{qde.2} reduces to a scalar differential equation of order $n+1$ in the unknown function $\varpi_0$. This scalar differential equation has at most $\binom{n}{2}$ apparent singularities.\qed
\end{thm}
We refer to the scalar differential equation of Theorem \ref{thsqde} as\footnote{In the terminology of \cite{Cot22}, for an arbitrary Dubrovin--Frobenius manifold, we use the terminology {\it master differential equation}.} the {\it scalar quantum differential equation}  of $X$.
\vskip1,5mm
For $(z,p)\in \widehat M_X^{\rm cyc}$, define the column vector $\overline{\varpi}$ as
\[\overline{\varpi}:=\left(\La^{-1}\right)^T\varpi.
\]
It satisfies
\begin{align}
\label{eq1.3}
\frac{\der \overline\varpi}{\der t^\al}&=\left(z\left(\La^{-1}\right)^T\mc C_\al^T\La^T+\frac{\der\left(\La^{-1}\right)^T}{\der t^\al}\La^T\right)\overline\varpi,\\
\label{qde.3}
\frac{\der \overline\varpi}{\der z}&=\left(\left(\La^{-1}\right)^T\mc U^T\La^T-\frac{1}{z}\left(\La^{-1}\right)^T\mu\La^T+\frac{\der\left(\La^{-1}\right)^T}{\der z}\La^T\right)\overline\varpi.
\end{align}
\begin{rem}
The entries of $\overline\varpi$ are the components of the $\Hat\nabla$-flat covector $\varpi$ computed with respect to the cyclic coframe $(\om_j)_{j=0}^n$.
\end{rem}
\begin{thm}\cite[Cor.\,2.27]{Cot22}
System \eqref{qde.3} is the companion system of the scalar quantum differential equation of $X$.\qed
\end{thm}

Denote by $ \eu O(\widetilde{\C^*}) $ the space of $\C$-valued holomorphic functions on the universal cover $ \widetilde{\C^*} $. For a fixed $ p \in M_X $, let $ \mc X_p $ be the $\C$-vector space of solutions to the differential equation \eqref{qde.2} specialized at $ p $. Define a morphism $ \nu_p\colon \mc X_p\to\eu O(\widetilde{\C^*}) $ by  
\beq\label{eqnupmap}
\varpi\mapsto\Phi_\varpi(z),\quad \Phi_\varpi(z):=z^{-\frac{\dim X}{2}}\bra \varpi(z,p), e(p)\ket.
\eeq

\begin{defn}\label{defmastfun}  
For a fixed $ p \in M_X $, the space of \emph{master functions} of $ X $ at $ p $ is defined as  
\[
\mc S_p(M) := {\rm im\,}\nu_p.
\]  
\end{defn}  

By Theorem \ref{thsqde}, the morphism $ \nu_p $ is injective for $ p \in M_X \setminus \mc A_\La $, see \cite[Th.\,2.31]{Cot22}.  

\begin{example}\label{qperiod}
A notable example of master function is the {\it quantum period} $ G_X(z) $, a generating function for genus-zero Gromov--Witten invariants of $ X $, extensively studied in \cite{CCGK16} and defined by
\beq\label{qper}
G_X(z) = 1 + \sum_{d \geq 2} \sum_{\substack{\beta \in \mathrm{NE}(X) \\ \langle \beta, c_1(X) \rangle = d}} z^d \int_{[X_{0,1,\beta}]^{\mathrm{vir}}} \mathrm{ev}^*(\mathrm{vol}_X) \cup \psi^{d-2},
\eeq
where $ \mathrm{vol}_X $ is a normalized volume form on $ X $, i.e.\,\,$\int_X{\rm vol}_X=1$. It follows from the results in Section~\ref{secjfuc} that $ G_X(z) $ is a master function at the point $ p = 0 $ of the quantum cohomology $ QH^{\mathrm{ev}}_X $.
\end{example}

\section{Master functions of Fano toric varieties}\label{sec3}
\subsection{Generalized Mittag--Leffler functions}Let $\mathsf M=(m^a_i)_{\substack{a=1,\dots, r\\ i=1,\dots, N}}$ be a $r\times N$-matrix with integer entries satisfying the positivity conditions
\beq\label{poscond}
\sum_{i=1}^Nm^a_i>0,\quad a=1,\dots,r.
\eeq
Consider the space $\C^r\times\N^r$, with coordinates $(\bm s,\bm d)=(s_1,\dots,s_r,d_1,\dots,d_r)$. Let $W_{\sf M}\subseteq \C^r\times\N^r$ be the set 
\[W_{\sf M}:=\left\{(\bm s,\bm d)\in\C^r\times\N^r\colon \sum_{a=1}^rm^a_i(d_a+s_a)\in\Z_{\leq -1}\right\}.
\]
Introduce the $\Gm$-{\it factor associated to $\sf M$} as the functions $\Gm_{\sf M}\colon W_{\sf M}\to\C$ as
\beq
\Gm_{\sf M}(\bm s,\bm d):=\prod_{i=1}^N\Gm\left(1+\sum_{a=1}^rm^a_i(d_a+s_a)\right).
\eeq Its reciprocal function $(\bm s,\bm d)\mapsto\Gm_{\sf M}(\bm s,\bm d)^{-1}$ is well defined on the whole $\C^r\times\N^r$, with zeroes at points of $W_{\sf M}$.

\begin{defn}\label{gML1}
The {\it generalized Mittag--Leffler function associated with $\sf M$} is the holomorphic function $\mc E_{\sf M}\in\eu O(\C^r\times\Tilde{\C^*})$ defined by the function series 
\beq\label{genML}
\mc E_{\sf M}(\bm s,z):=\sum_{\bm d\in\N^r}\frac{z^{\sum_{a=1}^r\sum_{i=1}^Nm^a_i(d_a+s_a)}}{\Gm_{\sf M}(\bm s,\bm d)}.
\eeq
\end{defn}

\begin{prop}
The following statements hold:
\begin{enumerate}
\item[(i)] The series \eqref{genML} is absolutely convergent for all $(\bm s,z) \in \C^r \times \widetilde{\C^*}$.
\item[(ii)] The rescaled function
\[
\widetilde{\mc E}_{{\sf M}}(\bm s,z) := z^{-\sum_{i=1}^N \sum_{a=1}^r m^a_i s_a} \mc E_{{\sf M}}(\bm s,z)
\]
is entire on $\C^{r+1}$.
\end{enumerate}
\end{prop}
\proof Point (ii) easily follows from point (i).
To prove convergence, we adapt a ``ratio test'' method originally introduced by J.\,Horn for double hypergeometric series \cite{Hor89}, generalized here to multiple series, see also \cite[pp. 223-229]{EMOF53}\cite{SD72}\cite[Ch.\,2.9]{Ext76}.
Consider a multiple power series
\beq\label{intermediatef} f(x_1,\dots,x_r)=\sum_{\bm d\in\N^r}A_{\bm d}\,x_1^{d_1}\dots x_r^{d_r},
\eeq
satisfying the {\it Horn's hypergeometric condition}: for any $j=1,\dots,r$ there exists a rational function $\phi_j(\bm d)\in\C(\bm d)$ such that  
\[\frac{A_{\bm d+\bm e_j}}{A_{\bm d}}=\phi_j({\bm d}),\qquad \text{with $\bm e_j=(0,\dots,1_j,\dots,0)$,}
\]whenever $A_{\bm d},A_{\bm d+\bm e_j}\neq 0$.
Following Horn, positive numbers $h_1,\dots,h_r$ will be called {\it associated radii of convergence} if this series is absolutely convergent for $|x_j|<h_j$, for all $j=1,\dots,r$, and divergent when $|x_j|>h_j$, for all $j=1,\dots,r$.  For any fixed $ \bm{x}_o \in \C^{r-1} $ and for each $ j = 1, \dots, r $,  consider the inclusion map
\[
i_{j,\bm{x}_o} \colon \C \to \C^r,\quad x \mapsto (x_{o,1}, \dots, x_{o,j-1}, x, x_{o,j}, \dots, x_{o,r-1}).
\]
Then the composition $ f \circ i_{j,\bm{x}_o} $ defines a power series in one variable with radius of convergence denoted by $ {\sf R}_{j,\bm{x}_o}(f) \in \R_{\geq 0} $. Define
\[
H_j := \sup_{\bm{x}_o \in \C^{r-1}} {\sf R}_{j,\bm{x}_o}(f),
\]
as the maximal possible radius of convergence in the $ x_j $-direction.
In the absloute $\bm h=(h_1,\dots, h_r)$ space, the points representing the associated radii convergence lie on a hypersurface $H$, entirely contained in the hyper-rectangle $[0,H_1]\,\times\dots\times\, [0,H_r]$, and dividing this hyper-rectagle into two parts of which the one containing $\bm h=0$ is the absolute value representation of the domain of convergence of the multiple series.

Define the asymptotic limits $\Phi_j(\bm d):=\lim_{t\to\infty}\phi_j(td_1,\dots,td_r)$ for any $j=1,\dots, r$. One can show that $H_j=|\Phi_j(\bm e_j)|^{-1}$, with $j=1,\dots,r$, and that $H$ admits the parametric form $h_j=|\Phi_j(\bm d)|^{-1}$, with $j=1,\dots,r$.

Fix $\bm s_o\in\C^r$, and set 
\beq\label{coeffspecial} 
A_{\bm d}:=\frac{1}{\Gm_{\sf M}(\bm s_o,\bm d)},\quad \bm d\in\N^r.
\eeq 
For any $j=1,\dots,r$, set 
\[\phi_j(\bm d):=\frac{A_{\bm d+\bm e_j}}{A_{\bm d}}=\frac{\Gm_{\sf M}(\bm s_o,\bm d)}{\Gm_{\sf M}(\bm s_o,\bm d+e_j)}=\frac{\prod_{i=1}^N\Gm(1+\sum_{a=1}^rm^a_i(d_a+s_{o,a}))}{\prod_{i=1}^N\Gm(1+m^j_i+\sum_{a=1}^rm^a_i(d_a+s_{o,a}))}.
\]By Stirling's formula, we have
\[\frac{\Gm(z)}{\Gm(z+b)}=z^{-b} \left(1+\frac{b-b^2}{2 z}+\frac{3 b^4-2 b^3-3 b^2+2 b}{24 z^2}+O\left(z^{-3}\right)\right),\quad z\to\infty,\quad |\on{arg}z|<\pi,
\]and we deduce
\[\phi_j(t\bm d)\sim c_j({\bm d})\cdot t^{-\sum_{i=1}^Nm^j_i},\quad c_j(\bm d):=\prod_{i=1}^N(\sum_{a=1}^rm^a_id_a)^{-m^j_i},\qquad t\to\infty.
\]The positivity conditions \eqref{poscond} imply $\Phi_j(\bm d) \equiv 0$ for all $j=1,\dots,r$. Hence, for any fixed $\bm s_o\in\C^r$, the series \eqref{intermediatef}, with coefficients \eqref{coeffspecial}, converges absolutely for all $x_i \in \C$. We conclude in virtue of the identity
\[\mc E_{{\sf M}}(\bm s,z)=z^{\sum_{a=1}^r\sum_{i=1}^Nm^a_is_a}\cdot f\left(z^{\sum_{i=1}^Nm^1_i},\dots, z^{\sum_{i=1}^Nm^r_i}\right).
\]
\endproof

\begin{example}\label{e1}
For $r=N=1$ and ${\sf M}=(1)$, we have the function
\beq
\mc E_{(1)}(s,z)=\sum_{k=0}^\infty\frac{z^{k+s}}{\Gm(1+k+s)},\quad (s,z)\in\C\times \Tilde{\C^*}.
\eeq
The function $\mc E_{(1)}(s,z)$ has been already introduced in \cite{Cot24}, and denoted by $\mc E(s,z)$. We have a remarkable relation with the B\"ohmer--Tricomi normalized incomplete Gamma function $\gm^*\in\eu O(\C^2)$:
\beq
\mc E(z,s)=e^zz^s\gm^*(s,z),\quad \gm^*(s,z):=\frac{z^{-s}}{\Gm(s)}\gm(s,z),\quad \gm(s,z):=\int_0^zt^{s-1}e^{-t}{\rm d}t.
\eeq 
For further properties of $\gm^*$, see \cite{B\"oh39,Tri50a,Tri50b,Tri54,Gau98}.
\end{example}

\begin{example}
More in general, for $r=N=1$ and ${\sf M}=(m)$, with $m\in\N^*$, we have the function 
\beq
\mc E_{(m)}(s,z)=\sum_{k=0}^\infty\frac{z^{mk+ms}}{\Gm(1+mk+ms)},\quad (s,z)\in\C\times \Tilde{\C^*}.
\eeq
The function $\mc E_{(m)}(s,z)$ is related to the classical Mittag--Leffler function 
\beq\label{ML1}
E_{\al,\bt}(z)=\sum_{k=0}^\infty\frac{z^k}{\Gm(\al k+\bt)},\quad\text{ with $\al,\bt\in\C,$ $\on{Re}\al>0$,}
\eeq by the identity
\[\mc E_{(m)}(s,z)=z^{ms}E_{m,1+ms}(z^m).
\]
By integration of the series \eqref{ML1}, we also have the integral representation
\beq
\mc E_{(m)}(s,z)=\int_0^zt^{ms-1}E_{m,ms}(t^m){\rm d}t.
\eeq
\end{example}

\begin{defn}\label{gML2}
Given a generalized Mittag--Leffler function $\mc E_{\sf M}(\bm s,z)$, we introduce the associated functions $\mc E_{{\sf M},\bm\al}\in\eu O(\Tilde{\C^*})$, with $\bm\al\in\N^r$, via the Taylor expansion
\beq
\mc E_{\sf M}(\bm s,z)=\sum_{\bm\al\in\N^r}\frac{\bm s^{\bm\al}}{\bm\al!}\mc E_{{\sf M},\bm\al}(z),\quad \bm s^{\bm\al}:=\prod_{j=1}^rs_j^{\al_j},\quad \bm \al!:=\prod_{j=1}^r\al_j!.
\eeq In other words, we have
\beq
\mc E_{{\sf M},\bm\al}(z):=\left.\der^{\bm\al}_{\bm s}\right|_{\bm s=0}\mc E_{\sf M}(\bm s,z),\quad \der^{\bm\al}_{\bm s}:=\frac{\der^{|\bm\al|}}{\der s_1^{\al_1}\dots\der s_r^{\al_r}},\quad |\bm\al|:=\al_1+\dots+\al_r.
\eeq
\end{defn}

\subsection{Mellin--Barnes integral representations.}\label{secMB} We work in the space $\C^r$ with coordinates $\bm{\zeta} = (\zeta_1, \dots, \zeta_r)$.  
For any $a \in \C$ and $\varepsilon \geq 0$, denote by $\overline{D}(a, \varepsilon) := \{ z \in \C : |z - a| \leq \varepsilon \}$ the closed disk in $\C$ centered at $a$ with radius $\varepsilon$.  
Given $\bm{a} \in \C^r$ and $\varepsilon > 0$, we define the closed polydisk
\[
\overline{\mathbb{D}}(\bm{a}, \varepsilon) := \prod_{j=1}^r \overline{D}(a_j, \varepsilon),
\]
centered at $\bm{a}$ with all polyradii equal to $\varepsilon$.  
The \emph{skeleton} $\tau(\bm{a}, \varepsilon)$ of the polydisk $\overline{\mathbb{D}}(\bm{a}, \varepsilon)$ is the real $r$-dimensional torus
\[
\tau(\bm{a}, \varepsilon) := \left\{ \bm{\zeta} \in \C^r : |\zeta_j - a_j| = \varepsilon \text{ for } j = 1, \dots, r \right\} = \prod_{j=1}^r \partial \overline{D}(a_j, \varepsilon).
\]
The standard orientation on $\C^r$ induces natural orientations on $\overline{\mathbb{D}}(\bm{a}, \varepsilon)$, its boundary $\partial \overline{\mathbb{D}}(\bm{a}, \varepsilon)$, and the skeleton $\tau(\bm{a}, \varepsilon)$.

Let $\Lambda_{{\sf M}} \subseteq \C^r$ denote the semigroup of points of the form
$\left( \left(\sum_{i=1}^N m^1_i\right) k_1, \dots, \left(\sum_{i=1}^N m^r_i\right) k_r \right)$, with $\bm{k} \in \Z_{\leq 0}^r.$
We consider the punctured space $\C^r \setminus \Lambda_{{\sf M}}$.  
Given the family of polydisks $(\mathbb{D}(0, n))_{n \in \N}$, this yields an inverse system of abelian groups:
\[
\dots \to H_\bullet(\C^r \setminus \Lambda_{{\sf M}}, \mathbb{D}(0, n+1)) \to H_\bullet(\C^r \setminus \La_{\sf M}, \mathbb{D}(0, n)) \to \dots,
\]
and we define the inverse limit
\[
\Tilde{H}_\bullet(\C^r \setminus \La_{\sf M}) := \varprojlim_n H_\bullet(\C^r \setminus \La_{\sf M}, \mathbb{D}(0, n)).
\]

Finally, we introduce the real $r$-dimensional cycle $\mathfrak{H} := \prod_{a=1}^r \mathfrak{H}_a \subseteq \C^r$, where each $\mathfrak{H}_a$ is a Hankel contour in the $a$-th copy of $\C$.  
Each $\mathfrak{H}_a$ starts at $-\infty - \sqrt{-1}\varepsilon$, loops counterclockwise around the origin, and ends at $-\infty + \sqrt{-1}\varepsilon$, for some small $\varepsilon > 0$, separating the sets $(\sum_{i=1}^N m^a_i)\Z_{\leq 0}$ and $(\sum_{i=1}^N m^a_i)\Z_{>0}$.

\begin{lem}\label{homocycle}
In $\Tilde H_r(\C^r\setminus\La_{\sf M})$, we have $\frak H=\sum_{\bm a\in\La_{\sf M}}\tau(\bm a,\eps)$ for sufficiently small $\eps>0$.
\end{lem}
\proof
Each factor $\frak H_a$, $a=1,\dots,r$ is homologous to an infinite sum of small circles centered at $(\sum_{i=1}^Nm^a_i)\Z_{\leq 0}$.
\endproof

\begin{thm}\label{thMBML}
Define the meromorphic function $\Ups_{\sf M}$ on $\C^r\times\C^r\times \Tilde{\C^*}$ by
\begin{multline}
\Ups_{\sf M}(\bm\zeta,\bm s,z):=\frac{1}{\prod_{a=1}^r\sum_{i=1}^Nm^a_i}\cdot\frac{\prod_{a=1}^r\Gm\left(\frac{\zeta_a}{\sum_{i=1}^Nm^a_i}\right)\Gm\left(1-\frac{\zeta_a}{\sum_{i=1}^Nm^a_i}\right)}{\prod_{i=1}^N\Gm\left(1-\sum_{a=1}^r\frac{m^a_i}{\sum_{j=1}^Nm^a_j}\zeta_a+\sum_{a=1}^rm^a_is_a\right)}\\
\times \exp\left({-\pi\sqrt{-1}\sum_{a=1}^r\frac{\zeta_a}{\sum_{i=1}^Nm^a_i}}\right)z^{-\sum_{a=1}^r\zeta_a+\sum_{a=1}^r\sum_{i=1}^Nm^a_is_a}.
\end{multline}
The generalized Mittag--Leffler function $\mc E_{\sf M}(\bm s,z)$ admits the integral representation
\beq\label{MBint}
\mc E_{\sf M}(\bm s,z)=\frac{1}{(2\pi\sqrt{-1})^r}\int_\frak H\Ups_{\sf M}(\bm\zeta,\bm s,z){\rm d}\zeta_1\dots{\rm d}\zeta_r,\quad (\bm s,z)\in\C^r\times\Tilde{\C^*}.
\eeq
\end{thm}

\proof
The r.h.s. of \eqref{MBint} is convergent for any $(\bm s,z)\in\C^r\times\Tilde{\C^*}$. To see this, factorize $|\Ups_{\sf M}(\bm\zeta,\bm s,z)|=(\prod_{a=1}^r\sum_{i=1}^Nm^a_i)^{-1}A_1A_2A_3$, where
\begin{align*}
A_1&=\left|\prod_{a=1}^r\Gm\left(\frac{\zeta_a}{\sum_{i=1}^Nm^a_i}\right)\Gm\left(1-\frac{\zeta_a}{\sum_{i=1}^Nm^a_i}\right)\right|,\\
A_2&=\left|\prod_{i=1}^N\Gm\left(1-\sum_{a=1}^r\frac{m^a_i}{\sum_{j=1}^Nm^a_j}\zeta_a+\sum_{a=1}^rm^a_is_a\right)\right|^{-1},\\
A_3&=\left| \exp\left(-\pi\sqrt{-1}\sum_{a=1}^r\frac{\zeta_a}{\sum_{i=1}^Nm^a_i}\right)z^{-\sum_{a=1}^r\zeta_a+\sum_{a=1}^r\sum_{i=1}^Nm^a_is_a}\right|,
\end{align*}
and set $\zeta_a=x_a\pm\eps\sqrt{-1}$, for $a=1,\dots,r$. 
We have
\[
A_1=(2\pi)^r\prod_{a=1}^r\left[2\cosh\frac{2\pi\eps}{\sum_{j=1}^Nm^a_j}
-2\cos\frac{2\pi x_a}{\sum_{j=1}^Nm^a_j}\right]^{-\frac{1}{2}},
\]
so that $A_1$ is positive and bounded in the limit $\bm\zeta\to\infty$ along $\frak H$, i.e.\,\,$x_a\to-\infty$ for $a=1,\dots,r$. Moreover, if we set $x_a=\la x_a^o$ for some fixed $x_a^o<0$, we have
\[
\log A_2=-\left(\sum_{a=1}^r|x_a^o|\right)\la\log\la+O(\la),\quad\quad \log A_3=O(\la),
\]for $\la\to+\infty$. So, the integrand $|\Ups_{\sf M}(\bm\zeta,\bm s,z)|$ is exponentially decaying for $\bm\zeta\to\infty$ along $\frak H$, for any fixed $(\bm s,z)\in\C^r\times\Tilde{\C^*}$.

The function $\Gm(t)\Gm(1-t)=\pi/\sin{(\pi t)}$ has simple poles at $t=k\in\Z$, with residues $\Res_{t=k}\Gm(t)\Gm(1-t)=\lim_{t\to k}\pi(t-k)/\sin(\pi t )=(-1)^k$. By Lemma \ref{homocycle}, the r.h.s. of \eqref{MBint} can be written as a sum of residues of $\Ups_{\sf M}(\bm z,\bm s,z)$ at points $\bm\zeta\in\La_{\sf M}$: 
\begin{multline*}
\frac{1}{(2\pi\sqrt{-1})^r}\int_\frak H\Ups_{\sf M}(\bm\zeta,\bm s,z){\rm d}\zeta_1\dots{\rm d}\zeta_r=\sum_{\bm p\in\La_{\sf M}}\frac{1}{(2\pi\sqrt{-1})^r}\int_{\tau(\bm p,\eps)}\Ups_{\sf M}(\bm\zeta,\bm s,z){\rm d}\zeta_1\dots{\rm d}\zeta_r\\
=\sum_{\bm p\in\La_{\sf M}}\frac{
z^{-\sum_{a=1}^rp_a+\sum_{a=1}^r\sum_{i=1}^Nm^a_is_a}}{\prod_{i=1}^N\Gm\left(1-\sum_{a=1}^r\frac{m^a_i}{\sum_{j=1}^Nm^a_j}p_a+\sum_{a=1}^rm^a_is_a\right)}\exp\left(-\pi\sqrt{-1}\sum_{a=1}^r\frac{p_a}{\sum_{i=1}^Nm^a_i}\right)\qquad\\
\times\frac{1}{\prod_{a=1}^r\sum_{i=1}^Nm^a_i}\prod_{a=1}^r\Res_{\zeta_a=p_a}\left[\Gm\left(\frac{\zeta_a}{\sum_{i=1}^Nm^a_i}\right)\Gm\left(1-\frac{\zeta_a}{\sum_{i=1}^Nm^a_i}\right)\right]\\
=\sum_{\bm k\in\N^r}\frac{z^{\sum_{a=1}^r\sum_{i=1}^Nm^a_i(k_a+s_a)}}{\prod_{i=1}^N\Gm(1+\sum_{a=1}^rm^a_i(k_a+s_a))}=\mc E_{\sf M}(\bm s,z).
\end{multline*}
\endproof

\subsection{Toric varieties, weight data, stability conditions}  \label{sectoric}
General references are \cite{Aud04, BP15, CLS11}.
A projective toric variety $ X $ is constructed from the following data:  
\begin{itemize}
\item[(1)] An algebraic torus $ \mathsf T^r \cong (\C^*)^r $, 
\item[(2)] A collection of integral vectors $ u_1, \dots, u_N $ in the weight lattice $ \Hom(\mathsf T^r, \C^*)\cong \Z^r $ (the {\it weight data});
\item[(3)] A real vector $ \om \in 
\Hom(\mathsf T^r, \C^*) \otimes \R $ (the {\it stability condition}).
\end{itemize}

The vectors $ u_1, \dots, u_N $ define a group monomorphism $ \mathsf T^r \to (\C^*)^N $, inducing an action of $ \mathsf T^r $ on $ \C^N $.  
Define the index set
\[
\mc A_\om := \left\{ I \subset \{1,\dots,N\} \;\middle|\; \om \in \sum_{i \in I} \R_{>0} u_i \right\}.
\]
The toric variety $ X $ is the GIT quotient
\[
X := \C^N /\!\!/_\om \mathsf T^r = \mc U_\om/\mathsf T^r, \qquad 
\mc U_\om := \C^N \setminus \bigcup_{I \notin \mc A_\om} \C^I,\qquad \C^I:=\{(z_1,\dots,z_N)\colon z_i=0\text{ if }i\notin I\}.
\]
We will assume that:
\begin{itemize}
\item[(I)] For all $ i $, the subset $ \{1,\dots,N\} \setminus \{i\} \in \mc A_\om $;
\item[(II)] For every $ I \in \mc A_\om $, the set $ \{u_i\}_{i \in I} $ spans $ \Hom(\mathsf T^r, \C^*) $ over $ \Z $;
\item[(III)] The only non-negative relation $ \sum_{i=1}^N c_i u_i = 0 $ is the trivial one: $ c_i = 0 $ for all $ i $.
\end{itemize}
These conditions imply that the variety $ X $ is non-empty, smooth, and compact. The variety $X$ has dimension $n=N-r$.

\begin{rem}\label{remtorus}
The injective morphism of groups $ \mathsf T^r \to (\C^*)^N $ fits into an exact sequence $1\to {\sf T}\to (\C^*)^N\to \mathbb T\to 1$, where $\mathbb T:=(\C^*)^N/{\sf T^r}\cong (\C^*)^{N-r}$ is a torus acting on $X_\om$ with open dense orbit.
\end{rem}

\begin{rem}
The strictly convex cone $\sum_{i=1}^N \R_{>0} u_i$ admits a chamber decomposition known as the \emph{secondary} (or {\it GKZ}) \emph{fan}. The walls of this decomposition are the $(r-1)$-dimensional cones spanned by proper subsets of the set $\{u_i\}_{i=1}^N$. The chambers are the $r$-dimensional open cones within $\sum_{i=1}^N \R_{>0} u_i$ defined as the connected components of the complement of the union of the walls. If two stability conditions $\omega_1$ and $\omega_2$ lie in the same chamber, then the corresponding GIT quotients $X_{\omega_1}$ and $X_{\omega_2}$ coincide.
\end{rem}

A character $ \rho \in \Hom(\mathsf T^r, \C^*) $ defines a line bundle $\mc L_\rho := \C \times_{\rho} \mc U_\om \rightarrow X$, where
\begin{equation}
\label{eq:linebundleonX}
\C \times_{\rho} \mc U_\om:= \C\times \mc U_\om/\sim,\quad (\rho(t)^{-1}v,z)\sim (v,tz),\quad t\in\mathsf T^r_\C.
\end{equation}
This gives a natural identification:
\[
\Hom(\mathsf T^r, \C^*) \cong \on{Pic}(X) \cong H^2(X, \Z),\qquad \rho\mapsto \mc L_\rho,\quad \mc L_\rho\mapsto c_1(\mc L_\rho).
\]
Under this identification, the vector $ u_i $ corresponds to the Poincaré dual of the torus-invariant divisor $ D_i:=\{z_i = 0\} $, $i=1,\dots, N$, where $ z_i $ is the standard coordinate on $ \C^N $. By the Danilov--Jurkiewicz Theorem, the cohomology $H^\bullet (X,\Z)$ is generated by the classes $u_1,\dots,u_N$, so that odd Betti numbers vanish (i.e.\,\,$H^\bullet (X,\C)=\eu H_X$). More precisely, $X$ is of Hodge--Tate type ($h^{p,q}(X)=0$, unless $p=q$). The Chern character and total Chern class of $X$ can be expressed in terms of $u_1,\dots, u_N$: from the generalized Euler exact sequence
\beq
0\to\eu O_X^{\oplus r}\to\bigoplus_{i=1}^N\eu O_X(u_i)\to\eu T_X\to 0,
\eeq we deduce
\beq\label{toriclasses} \on{ch}(X)=-r+\sum_{i=1}^Ne^{u_i},\qquad c(X)=\prod_{i=1}^N(1+u_i). 
\eeq 
The Kähler cone $ \mc K_X \subset H^{1,1}(X,\R)\cong H^2(X, \R) \cong \Hom(\mathsf T^r, \C^*) \otimes \R $ is given by
\begin{equation}
\label{eq:Kaehlercone}
\mc K_X = \bigcap_{I \in \mc A_\om} \sum_{i \in I} \R_{>0} u_i.
\end{equation}
\begin{rem}We can also describe $X$ as a symplectic reduction of $\mc U_\om$. if $\mathsf K^r \cong (\mathbb S^1)^r$ is the maximal compact subgroup of $\mathsf T^r$, the moment map $ \mu\colon \C^N \to \on{Lie}(\mathsf K^r)^\vee $ associated to the $ \mathsf K^r$-action is $\mu(z_1,\dots,z_N) = \sum_{i=1}^N \tfrac{1}{2} |z_i|^2 u_i,$
yielding a symplectic quotient description  $X \cong \mu^{-1}(\om)/\mathsf K^r$,
with $ \om\in\mc K_X $ corresponding to the class of the reduced symplectic form (K\"ahler form).
\end{rem}

\begin{rem}\label{remfan}
Toric varieties can also be described via fans. Given a toric variety $ X $ defined by data $ ({\sf T}^r, (u_i)_{i=1}^N, \omega) $, the characters $ u_i $ induce a map $ \Hom(\C^*, \mathsf{T}^r) \hookrightarrow \Hom(\C^*,(\C^*)^N)\cong\Z^N $, which fits into the exact sequence
\[
0 \to \Hom(\C^*, \mathsf{T}^r) \to \Z^N \to \mathfrak{N} := \Z^N / \Hom(\C^*, \mathsf{T}^r) \to 0,
\]
with $ \mathfrak{N} \cong \Z^{N-r} $. Let $ b_i \in \mathfrak{N} $ be the image of the standard basis vectors $ e_i \in \Z^N $. Then the fan $ \Sigma_\omega $ on $ \mathfrak{N}_\R := \mathfrak{N} \otimes \R $ has 1-cones generated by $ b_i $, and a cone $ \sum_{i \in I} \R_{\geq 0} b_i $ belongs to $ \Sigma_\omega $ if and only if $ \omega \in \sum_{i \notin I} \R_{>0} D_i $.
\end{rem}

\subsection{Master functions via generalized Mittag--Leffler functions} \label{secmstoric} Let $X$ be a Fano toric smooth projective variety associated with the data $({\sf T}^r,(u_i)_{i=1}^N,\om)$.
Since $h^{2,0}(X)=0$, we can fix a basis $(T_1,\dots, T_r)$ of $H^2(X,\Z)$ whose images in $H^2(X,\R)$ lie in the closure $\overline{\mc K_X}$, that is a {\it nef integral basis}. Each $ u_i $ expands as:
\begin{equation}
\label{eq:classesu}
u_i = \sum_{a=1}^r m_i^a T_a,\quad m^a_i\in\Z,\quad i=1,\dots,N.
\end{equation}
\begin{lem}\label{lemposcond}
The matrix $\mathsf M=(m^a_i)_{\substack{a=1,\dots, r\\ i=1,\dots, N}}$ satisfies the positivity conditions \eqref{poscond}.
\end{lem}
\proof
We have $c_1(TX)=\sum_{i=1}^Nu_i=\sum_{a=1}^r\sum_{i=1}^Nm^a_iT_a$ by equation \eqref{toriclasses}. The Fano assumption, together with ampleness criterion \eqref{ample} imply the result.
\endproof

The following result relates the space of master functions $\mc S_0(X)$ at the point $0\in Q\!H^{\rm ev}_X$ with the generalized Mittag--Leffler functions of Definitions \ref{gML1} and \ref{gML2}.

\begin{thm}\label{MTH0}
The space $ \mc S_0(X) $ of master functions is contained in the $ \C $-linear span of the generalized Mittag--Leffler functions $ \mc E_{{\sf M},\bm\alpha} $, with $ |\bm\alpha| = 0, \dots, \dim_\C X $. More precisely, there exist $ \Z $-linear combinations of the functions $ \mc E_{{\sf M},\bm\alpha} $ with $ |\bm\alpha| \leq \dim_\C X $ that form a $ \C $-basis of $ \mc S_0(X) $. In particular, the quantum period $ G_X $ equals the function $ \mc E_{{\sf M},\bm 0} $.
\end{thm}

The proof will be given in Section \ref{secproof1}.

Theorems \ref{MTH0} and \ref{thMBML} provide Mellin--Barnes type integral representations for bases of solutions to the quantum differential equations of $ X $. On the other hand, mirror symmetry offers alternative integral representations via oscillatory integrals.

If $ X $ is the Fano toric variety defined by the data $ (\mathsf{T}^r, (u_i)_{i=1}^N, \omega) $, we have the dual short exact sequence
\begin{equation} \label{seseqt}
1 \to \check{\mathbb{T}} \to (\C^*)^N \to \check{\mathsf{T}}^r \to 1,
\end{equation}
dual to the one in Remark \ref{remtorus}. The mirror Landau--Ginzburg model of $ X $ is given by the map $ \text{\textcyr{p}}: (\C^*)^N \to \check{\mathsf{T}}^r $, together with the potential $ W: (\C^*)^N \to \C $, $ \bm{x} \mapsto \sum_{i=1}^N x_i $.

Fixing a splitting of \eqref{seseqt}, the fibers $ \text{\textcyr{p}}^{-1}(q) $ are identified with $ \check{\mathbb{T}} $, and the restriction $ W_q:=W|_{\text{\textcyr{p}}^{-1}(q)} $ becomes a family of Laurent polynomials:
\[
W_q(\bm{t}) = q^{a_1} \bm{t}^{b_1} + \cdots + q^{a_N} \bm{t}^{b_N},\quad \bm t\in \check{\mathbb{T}},
\]
with $ a_i \in \Z^r $, $ b_i \in \Z^{N-r} $, the generators of the 1-cones of the fan $ \Sigma_\omega $ (see Remark \ref{remfan}).

Givental's Toric Mirror Theorem \cite{Giv94,Giv96}, and its further generalizations \cite{LLY97,LLY99a,LLY99b, Iri17},  establishes an isomorphism between the quantum cohomology algebra of $X$ and the Jacobi ring of $ W$, matching the Frobenius manifold structure of $ QH(X) $ with the Saito structure associated to $ W $. It also yields oscillatory integral representations of solutions to the quantum differential equation \eqref{eq1} of the form:
\beq\label{oscill}
\int_{\gamma \subset \text{\textcyr{p}}^{-1}(q)} e^{-z W} \varpi_q,
\eeq
where $ \gamma $ is a suitable cycle and $ \varpi_q $ is a holomorphic volume form on the fiber.

These oscillatory integrals are expected to be related to our Mellin--Barnes type integrals via Mellin transforms and a version of Gale duality. However, making this relation explicit -- e.g., identifying suitable cycles $ \gamma $ so that the integrals \eqref{oscill} correspond to the integrals \eqref{MBint} over the Hankel contours $ \mathfrak{H} $ from Section \ref{secMB} -- remains a challenging open problem. For an accessible overview of this connection, see the survey \cite{Iri19}. 

The reader may also consult \cite{TV23}, where the authors propose a new hypergeometric Landau--Ginzburg mirror symmetry model for partial flag varieties. In this framework, Jackson integral representations of Mellin--Barnes type are effectively promoted to Landau--Ginzburg models: the role of the superpotential is played by a product of Gamma functions multiplied by the exponential of a linear form, while the mirror space is interpreted as the complement in affine space of the pole locus of this function. Notably, this mirror is not an algebraic variety, but a complex analytic one. See also \cite{TV21,CV21,CV24}, where it is shown that, in the special cases of projective spaces and Grassmannians, Jackson integral representations of solutions to equivariant quantum differential equations are equivalent to genuine Mellin--Barnes integrals over suitable cycles.

\begin{rem}
In physics literature, Mellin--Barnes integrals of the type of Theorem \ref{thMBML} arise naturally in the exact formulas for D-brane central charges via hemisphere partition functions in 2d $(2,2)$ supersymmetric gauge theories; see \cite{HR13}.
\end{rem}

\section{Master functions of Fano complete intersections and projective bundles}\label{sec4}

\subsection{Analytic Borel--Laplace multitransforms}
Given $h,s\in\N^*$, let $\frak M_{s,h}$ be the set of $(s+1)\times h$-matrices 
\[\eu M=
\begin{pmatrix}
\eu M^1_1&\dots&\eu M^1_h\\
\eu M^2_1&\dots&\eu M^2_h\\
\vdots&\ddots&\vdots\\
\eu M^{s+1}_1&\dots&\eu M^{s+1}_h
\end{pmatrix},
\]with complex entries satisfying the conditions
\[
\eu M^1_j\neq0,\qquad \sum_{a=2}^{s+1}\eu M^a_j\neq 0,\qquad j=1,\dots,h.
\]
Fix a matrix
$$\eu M=
\begin{pmatrix}
\al_1&\dots&\al_h\\
\bt^1_1&\dots&\bt^1_h\\
\vdots&\ddots&\vdots\\
\bt^s_1&\dots&\bt^s_h
\end{pmatrix}\in\frak M_{s,h}.$$
\begin{defn}
The Borel $\eu M$-multitransform of an $h$-tuple $(\Phi_1,\dots, \Phi_h)\in\eu O(\Tilde{\C^*})^h$ of $\C$-valued functions is defined --\,whenever the integral converges\,-- as
\begin{multline}
\mathscr B_{\eu M}[\Phi_1,\dots,\Phi_h](z):=\\
\frac{1}{(2\pi\sqrt{-1})^s}\int_{\frak H}\dots\int_{\frak H}\left[\prod_{j=1}^h\Phi_j\left(z^{\frac{1}{\al_j\sum_{i=1}^s\bt^i_j}}\prod_{i=1}^s\zeta_i^{-\bt^i_j}\right)\right]\frac{e^{\sum_{i=1}^s\zeta_i}}{\zeta_1\dots\zeta_s}{\rm d}\zeta_1\dots {\rm d}\zeta_s
\end{multline}
where $\frak H$ is a Hankel-type contour of integration, see Figure \ref{gammahankel}.
\end{defn}

\begin{figure}[ht!]
\centering
\def\svgscale{1}
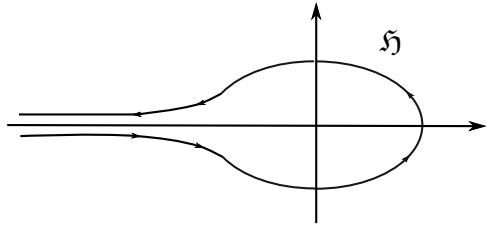
\caption{Hankel-type contour of integration defining Borel $\eu M$-multitransform.}
\label{gammahankel}
\end{figure}

\begin{defn}
The Laplace $\eu M$-multitransform of an $h$-tuple $(\Phi_1,\dots, \Phi_h)\in\eu O(\Tilde{\C^*})^h$ of $\C$-valued functions is defined --\,whenever the integral converges\,-- as
\begin{multline}
\mathscr L_{\eu M}\left[\Phi_1,\dots,\Phi_h\right](z):=\int_0^\infty\dots\int_0^\infty\left[\prod_{j=1}^h\Phi_j\left(z^{\al_j\sum_{i=1}^s\bt^i_j}\prod_{i=1}^s\zeta_i^{\bt^i_j}\right)\right]\exp\left(-\sum_{i=1}^s\zeta_i\right){\rm d}\zeta_1\dots{\rm d}\zeta_s.
\end{multline}
\end{defn}

The definitions of Borel--Laplace $\eu M$-multitransforms naturally extends to $h$-tuple of functions with values in a finite dimensional $\C$-algebra $A$.

\begin{prop}\cite{Cot22}
Let $(e_1,\dots, e_n)$ be a basis of $A$ and $\Phi_1,\dots,\Phi_h$ be $A$-valued functions. Write $\Phi_i=\sum_j\Phi_i^je_j$ for $\C$-valued component functions $\Phi_i^j$. The components of $\mathscr B_{\eu M}[\Phi_1,\dots,\Phi_h]$ (resp. $\mathscr L_{\eu M}[\Phi_1,\dots,\Phi_h]$) are $\C$-linear combinations of the $h\cdot n$ $\C$-valued functions $\mathscr B_{\eu M}[\Phi_1^{i_1},\dots,\Phi_h^{i_h}]$ (resp. $\mathscr L_{\eu M}[\Phi_1^{i_1},\dots,\Phi_h^{i_h}]$), where $(i_1,\dots, i_h)\in\left\{1,\dots,n\right\}^{\times h}$.\qed
\end{prop}

\subsection{Master functions of complete intersections as Laplace multitransforms} 
Let $ X_1,\dots,X_h $ be smooth Fano complex projective varieties with $ \det TX_j = L_j^{\otimes \ell_j} $ for some ample line bundles $ L_j \to X_j $ and $ \ell_j\in\mathbb{N}^* $, for all $j=1,\dots,h$.

Define $ Y $ as a smooth subvariety of $ X := \prod_{j=1}^{h} X_j $, given by the zero locus of a regular section of  
\[
E := \bigoplus_{i=1}^{s} \mathcal{L}_i, \quad \text{where } \mathcal{L}_i =\bigboxtimes_{j=1}^{h} L_j^{\otimes d_{ij}}, 
\]  
for  $ d_{ij} \in \mathbb{N}^* $ satisfying  
\[
\sum_{i=1}^{s} d_{ij} < \ell_j, \quad \text{for all } j=1,\dots,h.
\]  
This condition guarantees that $Y$ is Fano, by the adjunction formula. 

By the Künneth isomorphism, any element of $ H^2(X, \mathbb{C}) $ can be written as  
\[
\bm{\delta} = \sum_{j=1}^{h} 1 \otimes \dots \otimes \delta_j \otimes \dots \otimes 1,  
\quad \text{with } \delta_i \in H^2(X_i, \mathbb{C}).
\]  
Denote by $ \iota\colon Y \to X $ the inclusion.

 \begin{thm}\label{TH1}
 Let $\bm \delta\in H^2(X,\C)$, $\delta_i\in H^2(X_i,\C)$ as above, and $\mc S_{\delta_i}(X_i)$ the corresponding space of master functions of $QH^\bullet(X_i)$. There exists a rational number $c_{\bm\delta}\in\Q$ such that the space of master functions $\mc S_{\iota^*{\bm\delta}}(Y)$ is contained in image of the $\C$-linear map $\mathscr P_{(\bm \ell,\bm d)}\colon \otimes_{j=1}^h\mc S_{\delta_j}(X_j)\to \mc O(\widetilde{\C^*})$ defined by 
\[\mathscr P_{(\bm\ell,\bm d)}[\Phi_1,\dots,\Phi_h](z):=e^{-c_{\bm\delta} z}\mathscr L_{\eu M}[\Phi_1,\dots,\Phi_h](z),
\]
where \[\eu M=\begin{pmatrix}
\al_1&\dots&\al_h\\
\bt^1_1&&\bt^1_h\\
\vdots&\ddots&\vdots\\
\bt^s_1&\dots&\bt^s_h
\end{pmatrix},\qquad \al_j=\frac{\ell_j-\sum_{i=1}^sd_{ij}}{\sum_{i=1}^sd_{ij}},\qquad \bt^i_j=\frac{d_{ij}}{\ell_j},
\]for $i=1,\dots,s$ and $j=1,\dots,h$. In other words, any element of $\mc S_{\iota^*{\bm\delta}}(Y)$ is of the form
\beq\label{intth2}
e^{-c_{\bm\delta} z}\int_0^\infty\dots\int_0^\infty\prod_{j=1}^h\Phi_j\left(z^{\frac{\ell_j-\sum_{i=1}^sd_{ij}}{\ell_j}}\prod_{i=1}^s\zeta_i^\frac{d_{ij}}{\ell_j}\right)\exp\left(-\sum_{i=1}^s\zeta_i\right){\rm d}\zeta_1\dots{\rm d}\zeta_s,
\eeq
for some $\Phi_j\in\mc S_{\delta_j}(X)$ with $j=1,\dots,h$. Moreover, $c_{\bm\delta}\neq 0$ only if $\sum_{i=1}^sd_{ij}=\ell_j-1$ for some $j$.
 \end{thm}
 
 The proof will be given in Section \ref{secproof2}.
 
 \begin{rem}
Theorem~\ref{TH1} generalizes and interpolates both Theorem~7.2.1 and Theorem~7.3.1 of \cite{Cot22}. For $h=1$, that is when $ \eu M $ is taken as a column vector, one recovers the integral representations of the {\it first kind} in \cite[Thm.~7.2.1]{Cot22}. For $s=1$, when $ \eu M $ is a $2\times h$-matrix, the integral representations of the {\it second kind} in \cite[Thm.~7.3.1]{Cot22} are recovered.
\end{rem}

\subsection{Master functions of projective bundles as Borel multitransforms} Let $ X_1,\dots,X_h $ be smooth Fano complex projective varieties. Assume that $ \det TX_j = L_j^{\otimes \ell_j} $, with $ \ell_j\in\mathbb{N}^* $, for some ample line bundles $ L_j \to X_j $, where $ j=1,\dots,h $.  
\vskip1,5mm
Set $X:=\prod_{j=1}^hX_j$, and consider a vector bundle $E\to X$ of the form
\[
E := \bigoplus_{i=1}^{s} \mathcal{L}_i, \quad \text{where } \mathcal{L}_i = \bigboxtimes_{j=1}^{h} L_j^{\otimes (-d_{ij})}, 
\]  
for $ d_{ij} \in \mathbb{N}^* $ satisfying  
\[
\sum_{i=1}^{s} d_{ij} < \ell_j, \quad \text{for all } j=1,\dots,h.
\]  
This condition guarantees that the total space $P$ of the projective bundle $\pi\colon \Pb(\eu O_X\oplus E)\to X$ is Fano. As before, by the Künneth isomorphism, any element of $ H^2(X, \mathbb{C}) $ can be written as  
\[
\bm{\delta} = \sum_{j=1}^{h} 1 \otimes \dots \otimes \delta_j \otimes \dots \otimes 1,  
\quad \text{with } \delta_i \in H^2(X_i, \mathbb{C}).
\]  
Following the notations of \cite{Cot24}, for short we will denote by $\mc E(s,z)$ the generalized Mittag--Leffler function $\mc E_{(1)}(s,z)$ of Example \ref{e1}, and by $\mc E_k(z)$ the iterated partial derivatives $\der_s^k|_{s=0}\mc E_{(1)}$, with $k\geq 0$.

\begin{thm}\label{mt2}
The space $\mc S_{\pi^*\bm\dl}(P)$ is contained in the finite dimensional $\C$-vector space generated by the images of the maps $\mathscr P_{(\bm\ell,\bm d,k)}\colon\bigotimes_{j=1}^h\mc S_{\dl_j}(X_j)\to\mc O(\widetilde{\C^*})$ defined by
\[\Phi_1\otimes\dots\otimes\Phi_h\mapsto \mathscr B_{\eu M}\left[\Phi_1,\dots,\Phi_h,\mc E_k\right],
\]where $k=0,\dots,\dim_\C X+1$, and the matrix 
\[\eu M=\begin{pmatrix}
\al_1&\dots&\al_{h+1}\\
\bt^1_1&&\bt^1_{h+1}\\
\vdots&\ddots&\vdots\\
\bt^s_1&\dots&\bt^s_{h+1}
\end{pmatrix}\]
has entries
\[
\al_j = \frac{\ell_j^2}{(\sum_i d_{ij})[(\sum_i d_{ij}) - \ell_j]}, \qquad \al_{h+1} = \frac{1}{s(s+1)}, \qquad \bt^i_j = -\frac{d_{ij}}{\ell_j}, \qquad \bt^i_{h+1} = 1,
\]for $i=1,\dots,s$ and $j=1,\dots,h$. In other words, every element of $\mc S_{\pi^*\bm\dl}(P)$ is a finite sum of integrals of the form
\[\frac{1}{(2\pi\sqrt{-1})^s}\int_{\frak H}\dots\int_{\frak H}\left[\prod_{j=1}^h\Phi_j\left(z^{\frac{\ell_j-\sum_{i=1}^sd_{ij}}{\ell_j}}\prod_{i=1}^s\zeta_i^{\frac{d_{ij}}{\ell_j}}\right)\right]\mc E_k\left(\frac{z^{s+1}}{\zeta_1\dots\zeta_s}\right)\frac{e^{\sum_{i=1}^s\zeta_i}}{\zeta_1\dots\zeta_s}{\rm d}\zeta_1\dots {\rm d}\zeta_s,
\]with $\Phi_j\in\mc S_{\dl_j}(X_j)$, and $k=0,\dots,\dim_\C X+1$.
\end{thm}

 The proof will be given in Section \ref{secproof2}.

\begin{rem}
Theorem~\ref{mt2} generalizes Theorem~3.20 of \cite{Cot24}, the latter being the special case of $\eu M$ a $2\times h$-matrix (i.e.\,\,$s=1$).
\end{rem}

\section{Proofs of the main theorems}\label{sec5}

\subsection{Topological-enumerative solution and $ J $-function}\label{secjfuc}  Let $X$ be a smooth complex projective variety. Recall the notations of Section \ref{sec1}.

\begin{defn}  
Consider the functions $ \theta_{\beta,p}(z,\bm t) $ and $ \theta_{\beta}(z,\bm t) $, indexed by $ \beta=0,\dots, n $ and $ p\in\mathbb N $, defined as follows:  
\begin{gather}
\theta_{\beta,p}(\bm t):=\left.\frac{\partial^2\eu F_0^X(\bm t^\bullet)}{\partial t^{0,0}\partial t^{\beta,p}}\right|_{\substack{t^{\alpha, p}=0\text{ for }p>1,\\ t^{\alpha,0}=t^\alpha\text{ for }\alpha=0,\dots, n}},\\
\theta_\beta(z,\bm t):=\sum_{p=0}^\infty\theta_{\beta,p}(\bm t)z^p.
\end{gather}
Define the matrix $ \Theta(z,\bm t) $ by  
\begin{equation}
\Theta(z,\bm t)^\alpha_\beta:=\eta^{\alpha\lambda}\frac{\partial\theta_\beta(z,\bm t)}{\partial t^\lambda},\quad \alpha,\beta=0,\dots, n.
\end{equation}  
Furthermore, let $ R $ be the matrix representing the $\C$-linear operator  
\[
\eu H_X\to\eu H_X, 
\quad v\mapsto c_1(X)\cup v,
\]  
with respect to the basis $\bm T= (T_0,\dots, T_n) $. Using this, define the matrix $ Z_{\rm top}(z,\bm t) $ as  
\begin{equation}
Z_{\rm top}(z,\bm t) := \Theta(z,\bm t)z^\mu z^R.
\end{equation}  
\end{defn}

\begin{thm}\cite{Dub99,CDG20}  The matrix $ Z_{\rm top}(z,\bm t) $ forms a fundamental system of solutions to the joint system \eqref{eq1}, \eqref{qde}. \qed
\end{thm}

\begin{defn}  
The solution $ Z_{\rm top}(z,\bm t) $ is referred to as the \emph{topological-enumerative solution} of the system \eqref{eq1}, \eqref{qde}.  
\end{defn}

\begin{rem}
When $ X $ is Fano, the topological-enumerative solution can be uniquely identified among all fundamental solutions of the system \eqref{eq1}, \eqref{qde} by its behavior near $ z = 0 $. Indeed, if $ Z(z, \bm{t}) = H(z, \bm{t}) z^\mu z^R $ is such that the conjugated matrix $ z^{-\mu} H(z, \bm{t}') z^\mu $, with $ \bm{t}' = (0, t^1, \dots, t^r, 0, \dots, 0) $, is holomorphic at $ z = 0 $ and has expansion $ \exp(\delta \cup) + O(z) $, where $ \delta = \sum_{i=1}^r t^i T_i $, then $ H(z, \bm{t}) = \Theta(z, \bm{t}) $. This provides an effective criterion for computing the canonical solution. See \cite[Prop.\,7.2]{CDG20}.
\end{rem}

\vskip2mm  
Let $ \hbar $ be a formal parameter.  

\begin{defn}  
The \emph{$ J $-function} of $ X $ is an element of $ H^\bullet(X,\La_{X,\om})[\![\hbar^{-1}]\!] $, depending on $ \bm\tau\in H^\bullet(X,\C) $, and is defined as  
\[
J_X(\bm\tau):=1+\sum_{\al,\la=0}^n\sum_{p=0}^\infty\hbar^{-(p+1)}T_\la\eta^{\al\la}\left.\frac{\der^2\mc F^X_0}{\der t^{0,0}\der t^{\al,p}}\right|_{\substack{t^{\alpha, p}=0\text{ for }p>1,\\ t^{\alpha,0}=\tau^\alpha\text{ for }\alpha=0,\dots, n}}.
\]
\end{defn}

The restriction of the $ J $-function to the small quantum locus, i.e., when $ \bm\tau $ is constrained to $ H^2(X,\C) $, simplifies as follows.

\begin{lem}\cite{CK99,Cot22}\label{sjf}  
For $ \delta\in H^2(X,\C) $, we have  
\[
\pushQED{\qed}
J_X(\delta)=e^{\frac{\delta}{\hbar}}\left(1+\sum_{\al=0}^n\sum_{\bt\neq 0}\sum_{p=0}^\infty e^{\int_\bt\delta}\langle\tau_pT_\al,1\rangle^X_{0,2,\bt}T^\al\hbar^{-(p+1)}{\bf Q}^\bt\right). \qedhere
\popQED
\]
\end{lem}

\begin{thm}{\cite[Thm.\,5.2, Cor.\,5.3]{Cot22}}  \label{masJ}
Let $ \delta\in H^2(X,\C) $. For $ \alpha=0,\dots, n $, the $ (0,\alpha) $-entry of the matrix $ \eta Z_{\rm top}(z,\delta) $ equals  
\[
z^{\frac{\dim X}{2}}\int_XT_\alpha\cup J_X(\delta+\log z\cdot c_1(X))\rqh.
\]  
In particular, the components of  
\[
J_X(\delta+\log z\cdot c_1(X))\rqh,
\]  
expressed in any $\C$-basis of $ H^\bullet(X,\C) $, generate the space of master functions $ \mc S_\delta(X) $. \qed  
\end{thm}

\begin{rem}\label{remgx}
The quantum period $G_X(z)$ of Example \ref{qperiod} is the component along $T_0=1$ of the function $J_X(\log z\cdot c_1(X))\rqh$.
\end{rem}

\subsection{Proof of Theorem \ref{MTH0}}\label{secproof1} We are now ready to prove the first main result. The Toric Mirror Theorem of A.\,Givental \cite{Giv94,Giv96} gives an explicit formula for the $J$-function of a Fano toric variety $X$: in the notations of Sections \ref{sectoric} and \ref{secmstoric}, we have
\beq\label{GivTMT}
J_X(\dl)=e^{\dl/\hbar}\sum_{\bt\in {\rm NE}(X)}e^{\int_\bt\dl}\frac{\prod_{i=1}^N\prod_{m=-\infty}^0(u_i+m\hbar)}{\prod_{i=1}^N\prod_{m=-\infty}^{\langle u_i,\bt\rangle}(u_i+m\hbar)},\quad \dl\in H^2(X,\C),
\eeq
where the cohomology classes $u_1,\dots,u_N\in H^2(X,\Z)$ represent the torus invariant divisors $D_1,\dots, D_N$. See also \cite{LLY97,LLY99a,LLY99b, Iri17}.

\begin{lem}\label{lemhbar}
For any indeterminate $x$ and any $a\in\Z$, we have
\beq\label{lemgm}\frac{\prod_{m=-\infty}^0(x+m\hbar)}{\prod_{m=-\infty}^a(x+m\hbar)}=\hbar^{-a}\frac{\Gm(\frac{x}{\hbar}+1)}{\Gm(\frac{x}{\hbar}+a+1)},\quad \text{identically in $x$,}
\eeq the r.h.s.\,\,being interpreted as Taylor expansion at $x=0$.
\end{lem}
\proof Let $a\geq 0$. The r.h.s.\,\,of \eqref{lemgm} equals $\prod_{m=1}^a(x+m\hbar)^{-1}$. We have
\begin{multline*}
\Gm\left(\frac{x}{\hbar}+a+1\right)=\left(\frac{x}{\hbar}+a\right)\Gm\left(\frac{x}{\hbar}+a\right)=\dots=\left(\frac{x}{\hbar}+a\right)\left(\frac{x}{\hbar}+a\right)\dots\left(\frac{x}{\hbar}+1\right)\Gm\left(\frac{x}{\hbar}+1\right)\\=\hbar^{-a}\prod_{m=1}^a(x+m\hbar)\Gm\left(\frac{x}{\hbar}+1\right).
\end{multline*}
This proves the claim for $a\geq 0$. The case $a<0$ is similar.
\endproof
As in Section \ref{secmstoric}, introduce a nef integral basis ${\bf T}=(T_1,\dots, T_r)$ of $H^2(X,\Z)$, so that $u_i=\sum_{a=1}^rm^a_iT_a$, $i=1,\dots, N$, for a coefficients matrix $\sf M$ satisfying the positivity conditions \eqref{poscond} (by Lemma \ref{lemposcond}). For any multi-index $\bm\al\in\N^r$, set ${\bf T}^{\bm\al}:=\prod_{i=1}^rT_i^{\al_i}$. Also, introduce the characteristic Gamma class $\Hat\Gm^+_X$ of $X$ via the identity
\[\Hat\Gm^+_X:=\prod_{i=1}^N\Gm(1+u_i).
\] 

From the identity \eqref{GivTMT}, and Lemma \ref{lemhbar}, we deduce\footnote{For any $\bt\in {\rm NE}(X)$, we have $\int_\bt T_a\geq 0$ for any $a=1,\dots,r$, the cone $\overline{K_X}$ and $\overline{{\rm NE}}(X)$ being dual. Hence the sum $\sum_{\bt}$ in \eqref{GivTMT} can be rerranged as a sum $\sum_{\bm d\in\N^r}$.}
\begin{multline*}
J_X(\log z\cdot c_1(X))\rqh=\Hat\Gm^+_X\cup\sum_{\bm d\in\N^r}\frac{z^{\sum_{a=1}^r\sum_{i=1}^Nm^a_i(d_a+T_a)}}{\prod_{i=1}^N\Gm(1+\sum_{a=1}^rm^a_i(d_a+T_a))}\\=\Hat\Gm^+_X\cup \mc E_{\sf M}(T_1,\dots, T_r,z)=\Hat\Gm^+_X\cup\sum_{\bm\al\in\N^r}\frac{{\bf T}^{\bm\al}}{\bm\al!}\mc E_{\sf M,\bm\al}(z).
\end{multline*}
We can extract a basis of $H^\bullet(X,\Z)$ from the family $\left({\bf T}^{\bm\al}\right)_{\bm\al\in\N^r}$, by the Danilov--Jurkiewicz Theorem. Hence, we can extract a $\C$-basis of $H^\bullet(X,\C)$ from the family $\left(\Hat\Gm^+_X\cup{\bf T}^{\bm\al}\right)_{\bm\al\in\N^r}$, the endomorphism $\Hat\Gm^+_X\cup(-)\in\End_\C(H^\bullet(X,\C))$ being invertible. Moreover, any nontrivial vanishing linear combination of the $\left(\Hat\Gm^+_X\cup{\bf T}^{\bm\al}\right)_{\bm\al\in\N^r}$ can be taken with integer coefficients, since $$a\mapsto\Hat\Gm^+_X\cup a=a+\text{terms of higher cohomological degree}.$$ 
We conclude by invoking Theorem \ref{masJ}, and Remark \ref{remgx}.
\subsection{Twisted Gromov--Witten theory and quantum Lefschetz}
For a holomorphic vector bundle $E \to X$, the moduli space of stable maps $X_{g,n,\bt}$ carries a locally free orbi-sheaf complex  
\beq\label{cpxmodsp}
0 \to E_{g,n,\bt}^0 \to E_{g,n,\bt}^1 \to 0,
\eeq
whose cohomology sheaves are  $R^0{\rm ft}_{n+1,*}({\rm ev}_{n+1}^*E)$, and $R^1{\rm ft}_{n+1,*}({\rm ev}_{n+1}^*E).$
Here, the morphisms ${\rm ft}_{n+1}\colon X_{g,n+1,\bt}\to X_{g,n\bt}$ and ${\rm ev}_{n+1}\colon X_{g,n+1,\bt}\to X$ are the forgetful and evaluation maps at the last marked point. The \emph{obstruction $K$-class} is defined as  
\[
E_{g,n,\bt}:=[E_{g,n,\bt}^0]-[E_{g,n,\bt}^1] \in K^0(X_{g,n,\bt}),
\]  
and it does not depend on the choice of the complex \eqref{cpxmodsp}. See \cite{CG07}.

Given a holomorphic vector bundle $E\to X$ and an invertible multiplicative characteristic\footnote{\,\,We say that $\bm c$ is invertible if $\bm c(E)\in H^\bullet(X,\C)$ is invertible for any vector bundle $E$. We say that $\bm c$ is multiplicative if $\bm c(E_1\oplus E_2)=\bm c(E_1)\bm c(E_2)$. Examples of multiplicative characteristic classes include the total Chern class and the Euler class.} class $\bm c$, the $(E,\bm c)$-{\it twisted Gromov-Witten invariants} of $X$ are defined as  
\[
\langle\tau_1^{d_1}\alpha_1\otimes \dots\otimes \tau_n^{d_n}\al_n\rangle_{g,n,\beta}^{X, E,\bm c}:=\int_{[X_{g,n,\bt}]^{\rm vir}}\bm c(E_{g,n,\bt})\cup \prod_{j=1}^n\psi_j^{d_j}\cup{\rm ev}^*_j(\alpha_j),\qquad \al_i\in \eu H_X.
\]  When $\bm c$ is trivial, these invariants reduce to the usual Gromov-Witten invariants.  
Introducing a fiberwise $\C^*$-action on $E$ via scalar multiplication renders the characteristic class $\bm c = \bm e$ (the $\C^*$-equivariant Euler class) invertible over\footnote{\,\,Here, $\lambda$ denotes the first Chern class of the line bundle $\mathcal{O}(1)$ over $\mathbb{CP}^\infty = B\C^*$.} $\mathbb{Q}(\lambda)$, the fraction field of $H^\bullet_{\C^*}({\rm pt}) \cong \mathbb{Q}[\lambda]$. This enables the definition of Euler-twisted Gromov-Witten invariants.

\begin{thm}\label{QLTH}\cite{CG07,Coa14}
If $E$ is convex and $Y$ a smoth zero locus of a regular section of $E$, then there exists the non-equivariant limit  $J_{E,\bm e}|_{\lambda=0}$ of the Euler-twisted J-function
\[
J_{E,\bm e}(\bm\tau)=1+\sum_{\alpha,k,n,\beta}\hbar^{-n-1}\frac{{\bf Q}^\beta}{k!}\langle\tau_nT_\alpha,1,\bm\tau,\dots,\bm\tau\rangle^{X,E,\bm e}_{0,k+2,\beta}T^\alpha\in H^\bullet(X,\La_{X,\om}[\la])[\![\hbar^{-1}]\!],
\]and moreoveor
\beq\label{QLTH+}
\iota^* J_{E,\bm e}(\bm\tau)|_{\lambda=0} \stackrel{\iota_*}{=} J_Y(\iota^*\bm\tau),\qquad \iota\colon Y\xhookrightarrow{} X.
\eeq For $E = \bigoplus_{i=1}^s L_i$ with nef $L_i$ and $c_1(E) \leq c_1(X)$, the non-equivariant limit $J_{E,\bm e}|_{\lambda=0}$, evaluated at $\dl\in H^2(X,\C)$, is computed from the hypergeometric modification $I_{X,Y}$ of $J_X=\sum_\bt J_\bt{\bf Q}^\bt$:  
\beq\label{Ifun}
I_{X,Y}(\delta) = \sum_{\beta} J_{\beta}(\delta) {\bf Q}^\beta \prod_{i=1}^s\prod_{m=1}^{\langle c_1(L_i),\beta\rangle}(c_1(L_i)+m\hbar).
\eeq 
This satisfies the asymptotics  
\[
I_{X,Y}(\delta) = F(\delta) + \frac{1}{\hbar} G(\delta) + O\left(\frac{1}{\hbar^2}\right),
\]  
where $F(\delta) \in H^0(X,\Lambda_X)$ and $G(\delta) \in H^0(X,\Lambda_X) \oplus H^2(X,\Lambda_X)$, leading to  
\[
J_{E,\bm e}(\phi(\delta))|_{\lambda=0} = \frac{I_{X,Y}(\delta)}{F(\delta)}, \quad \phi(\delta) := \frac{G(\delta)}{F(\delta)}.
\]  
If $c_1(X) > c_1(E)$, then  
\[
F(\delta) \equiv 1, \quad G(\delta) = \delta + H(\delta) \cdot 1,\quad H(\dl)=\sum_\bt\left(w_\bt{\bf Q}^\bt e^{\int_\bt\dl}\right).\cdot \dl_{1,\langle\bt, c_1(X)-c_1(E)\rangle},
\]  
for suitable rational numbers $w_\bt\in\Q$.\qed
\end{thm}

\subsection{The Elezi--Brown theorem} 
Let $ E=\bigoplus_{j=1}^s \mc L_j $ be a split vector bundle over $ X $, with line bundles $ L_j\to X $ satisfying:  
\begin{enumerate}
\item $ \mc L_j^* $ is ample for $ j=1,\dots,s $;
\item $ c_1(X)+c_1(E) $ is ample.
\end{enumerate}

Define $ P:=\Pb(\eu{O}_X\oplus E) $, with projection $ \pi\colon P\to X $ and sections $ \si_i\colon X\to P $, $i=0,\dots,s$, defined by the summands $\eu O_X,\mc L_1,\dots,\mc L_s$. 
Set
\[\xi:=c_1(\eu O_P(1)),\quad \text{where }\eu O_p(-1)\text{ is the tautological line bundle on }P.
\]The classical cohomology $ H^\bullet(P,\C) $ is an algebra over $ H^\bullet(X,\C) $ via pullback. By the Leray--Hirsch Theorem, $ \pi^*\colon H^\bullet(X,\C)\to H^\bullet(P,\C) $ is a ring monomorphism, giving the presentation  
$ H^\bullet(P,\C) \cong H^\bullet(X,\C)[\xi] / (\xi^{r+1}+c_1(V)\xi^r+\dots+c_{r+1}(V)) $.  
Moreover, the $\C$-linear map $ H^\bullet(X,\C)^{\oplus (r+1)} \to H^\bullet(P,\C) $, $ (\al_0,\dots,\al_r) \mapsto \sum_i\xi^i\pi^*\al_i $, is an isomorphism, yielding  
$ H^\bullet(P,\C) \cong \bigoplus_{i=0}^r\xi^i H^\bullet(X,\C) $ as $\C$-vector spaces. See e.g.\,\cite[Ch.\,4\S 20]{BT}\cite[Sec.\,9.3]{EH16}.

\begin{lem}[{\cite[Lem.\,1.0.1]{Ele05}}]\label{lemc}
If $ \mc L_i^* $ is nef for $ i=1,\dots,s $, then:
\begin{enumerate}
\item If $ (T_1,\dots, T_k) $ is a nef basis of $ H^2(X,\Q) $, then $ (T_1,\dots, T_k,\xi) $ is a nef basis of $ H^2(P,\Q) $.
\item The Mori cones satisfy  
\[
{\rm NE}(P)={\rm NE}(X)\oplus \Z_{\geq 0}\cdot [\ell],
\]  
where $ [\ell] $ is the class of a fiber line of $ \pi $, and $ {\rm NE}(X) $ is embedded in $ {\rm NE}(P) $ via $ \si_0 $.\qed
\end{enumerate}
\end{lem}

The small $ J $-function of $ X $ has the form  
\[
J_X(\dl)=e^\frac{\dl}{\hbar}\sum_{\bt\in H_2(X,\Z)}J^X_\bt(\dl){\bf Q}^\bt.
\]  
For $ \bt\in H_2(X,\Z) $ and $ \nu\in\N $, define the \textit{twisting factor}  
\[
\mc T_{\nu,\bt}:=\prod_{i=0}^s\frac{\prod_{m=-\infty}^0\left(\xi+\pi^*c_1(\mc L_i)+m\hbar\right)}{\prod_{m=-\infty}^{\nu+\langle c_1(\mc L_i),\bt \rangle}\left(\xi+\pi^*c_1(\mc L_i)+m\hbar\right)},\qquad \mc L_0:=\eu O_X,
\]  
and the \textit{small $ I $-function} (on $ H^2(P,\C)\cong H^2(X,\C)\oplus\,\C\xi $) via  
\[
I_P(\pi^*\dl+t\xi):=\exp\left(\frac{\pi^*\dl+t\xi}{\hbar}\right)\sum_{\bt}\sum_{\nu\geq 0}\pi^*J^X_\bt(\dl)\cup \mc T_{\nu,\bt}\,{\bf Q}^{\bt+\nu},\quad t\in\C.
\]  
Here, $ \bt+\nu $ denotes $ \bt+\nu[\ell] $ in $ {\rm NE}(P) $, as in Lemma \ref{lemc}.  

The following theorem, conjectured and partially proven by A.\,Elezi \cite{Ele05,Ele07}, was later established in full generality by J.\,Brown \cite{Bro14}.

\begin{thm}\label{EBT}
If $ L_j $ satisfy assumptions (1), (2), (3), then  
\beq\label{i=j}
J_P(\pi^*\dl+t\xi)=I_P(\pi^*\dl+t\xi).\qedhere
\eeq
\end{thm}

\begin{rem}
Brown's result \cite{Bro14} is more general, extending to toric fiber bundles over $ X $, but is stated in terms of A.\,Givental's Lagrangian cone formalism \cite{Giv04,CG07} rather than the explicit equality \eqref{i=j}.
\end{rem}

\subsection{The Ribenboim's algebras \texorpdfstring{$\mathscr F_{\bm\kappa}(A)$}{}}
Let $(A,+,\cdot,1_A)$ be an associative, commutative, unital and finite dimensional $\C$-algebra. Denote by ${\rm Nil(A)}$ the nilradical of $A$, that is 
\[{\rm Nil}(A):=\{a\in A\colon\exists \,n\in\N\,\,\text{s.t.\,\,}a^n=0\},
\]and set  $\N_A:=\{n\cdot 1_A\colon n\in\N\}$.
\vskip2mm
Let $h\in\N^*$, and fix a tuple $\bm \kappa:=(\kappa_1,\dots,\kappa_h)\in(\C^*)^h$. Define the monoid $M_{A,{\bm\kappa}}$ as the (external) direct sum of monoids
\[M_{A,{\bm\kappa}}:=\left(\bigoplus_{j=1}^h\kappa_j\mathbb N_A\right)\oplus{\rm Nil}(A).\]
If $x\in M_{A,\bm\kappa}$, denote by $x'$ its ``nilpotent part'', i.e. the projection of $x$ onto ${\rm Nil}(A)$. 
\vskip2mm
We have two maps $\nu_{\bm\kappa}\colon M_{A,{\bm\kappa}}\to \mathbb N^h$ and $\iota_{\bm\kappa}\colon M_{A,{\bm\kappa}}\to A$ defined by
\[\nu_{\bm\kappa}((\kappa_in_i1_A)_{i=1}^h,r):=(n_i)_{i=1}^h,\quad \iota_{\bm\kappa}((\kappa_in_i1_A)_{i=1}^h,r):=\sum_{i=1}^h\kappa_in_i1_A+r.
\]The natural inclusions $M_{A,\kappa_i}\to M_{A,{\bm\kappa}}$ induce a unique morphism
\[\rho_{\bm\kappa}\colon \bigoplus_{i=1}^hM_{A,\kappa_i}\to M_{A,\bm\kappa},
\]by universal property of the direct sums of monoids.

On $M_{A,{\bm\kappa}}$ we can define the partial order 
\[x\leq y \quad\text{iff}\quad x'=y'\text{ and }\nu_{\bm\kappa}(x)\leq \nu_{\bm\kappa}(y),
\]the order on $\mathbb N^h$ being the lexicographical one. This order makes $(M_{A,{\bm\kappa}},\leq)$ a strictly ordered monoid, that is 
\[\text{if } a,b\in M_{A,\bm\kappa}\text{ are such that }a<b,\,\,\text{then } a+c<b+c\quad\text{for all } c\in M_{A,\bm\kappa}.
\]
Define $\mathscr F_{\bm\kappa}(A)$ to be the set of all functions $f\colon M_{A,{\bm\kappa}}\to A$ whose support
\[{\rm supp}(f):=\left\{a\in M_{A,{\bm\kappa}}\colon f(a)\neq 0\right\}
\]is 
\begin{enumerate}
\item {\it Artinian}, i.e.\,\,every subset of ${\rm supp}(f)$ admits a minimal element,
\item and {\it narrow}, i.e.\,\,every subset of ${\rm supp}(f)$ of pairwise incomparable elements is finite.
\end{enumerate}
\begin{lem}\label{supplem}
For any $f\in\mathscr F_{\bm\kappa}(A)$, we have $\on{card}\on{supp}(f)\leq \aleph_0$. 
\end{lem}
\proof
Given $f\in\mathscr F_{\bm\kappa}(A)$, set $I_f:=\{r\in{\rm Nil}(A)\colon r=a'\text{ for some }a\in\on{supp}(f)\}$. For each $r\in I_f$ define $\mc I_{f,r}:=\{a\in\on{supp}(f)\colon a'=r\}$. Since each $\mc I_{f,r}$ is at most countable (by definition of $M_{A,\bm \kappa}$), and $\on{supp}(f)=\coprod_{r\in I_f}\mc I_{f,r}$, we get $\on{card}\on{supp}(f)\leq \aleph_0\cdot \on{card} I_f$. A choice function  $c\colon I_f\to\coprod_{r\in I_f}\mc I_{f,r}$ is necessarily injective, as the $\mc I_{f,r}$ are disjoint, and its image consists of pairwise incomparable elements. By the narrowness condition, $\on{card} I_f=\on{card} c(I_f)<\aleph_0$, implying $\on{card}\on{supp}(f)\leq \aleph_0$.
\endproof
The set $\mathscr F_{\bm \kappa}(A)$ is an $A$-module with respect to pointwise addition and multiplication of $A$-scalars.
We will denote the element $f\in\mathscr F_{\bm \kappa}(A)$ by
\[f=\sum_{a\in M_{A,\bm\kappa}}f(a)Z^a,
\]where $Z$ is an indeterminate. Given $f_1,f_2\in\mathscr F_{\bm \kappa}(A)$, define
\[f_1\cdot f_2:=\sum_{s\in M_{A,\bm\kappa}}\left(\sum_{(p,q)\in X_s(f,g)}f_1(p)\cdot f_2(q)\right)Z^{s},
\]where we set
\[X_s(f,g):=\left\{(p,q)\in M_{A,\bm\kappa}\times M_{A,\bm\kappa}\colon p+q=s,\quad f_1(p)\neq 0,\quad f_2(q)\neq 0\right\}.
\]

The following result is a consequence of P.\,Ribenboim's theory of generalized power series \cite{Rib92,Rib94}.
\begin{thm}
The product above is well-defined. The set $\mathscr F_{\bm \kappa}(A)$ is equipped with an $A$-algebra structure with respect to the operations above.\qed
\end{thm}

\begin{defn}
Let $r_o\in{\rm Nil}(A)$. We say that an element $f\in \mathscr F_{\bm \kappa}(A)$ is \emph{concentrated at $r_o$} if 
\[{\rm  supp}(f)\subseteq \left(\bigoplus_{i=1}^h\kappa_i\N_A\right)\times\{r_o\}.
\]
\end{defn}

\subsection{Formal Borel--Laplace multitransforms}\label{formal}
Let $s,h\in\N^*$. 
Given an $h$-tuple $\bm \kappa=(\kappa_1,\dots,\kappa_h)\in(\C^*)^h$, and a matrix
\beq\label{Matrix}\eu M=
\begin{pmatrix}
\al_1&\dots&\al_h\\
\bt^1_1&\dots&\bt^1_h\\
\vdots&\ddots&\vdots\\
\bt^s_1&\dots&\bt^s_h
\end{pmatrix}\in\frak M_{s,h},
\eeq we define two products $\eu M\wedge{\bm\kappa}$ and $\eu M\vee\bm\kappa$ by
\begin{align}\eu M\wedge\bm\kappa=\left(\frac{\kappa_1}{\al_1\sum_{a=1}^s\bt^a_1},\dots,\frac{\kappa_h}{\al_h\sum_{a=1}^s\bt^a_h}\right)\in(\C^*)^h,\\
\eu M\vee\bm\kappa=\left({\kappa_1}{\al_1\sum_{a=1}^s\bt^a_1},\dots,{\kappa_h}{\al_h\sum_{a=1}^s\bt^a_h}\right)\in(\C^*)^h.
\end{align}

\begin{defn}\label{formalF}
Let $ F\in\mathbb{C}[\![x]\!] $ be a formal power series given by  
\[
F(x)=\sum_{k=0}^\infty a_kx^k.
\]  
For $ \alpha\in{{\rm Nil}(A)} $, define $ F(\alpha)\in A $ as the finite sum  
\[
F(\alpha)=\sum_{k=0}^\infty a_k\alpha^k.
\]  
If $ F $ is invertible, i.e., if $ a_0\neq 0 $, then $ F(\alpha) $ is invertible in $ A $.  
\end{defn}
In what follows, we will typically consider $ F(x) $ as the Taylor series at $ x=0 $ of either the (shifted) Euler Gamma function $ \Gamma(\lambda+x) $, with $ \lambda\in\mathbb{C}\setminus\mathbb{Z}_{\leq 0} $, or the (shifted) reciprocal Euler Gamma function $ \frac{1}{\Gamma(\lambda+x)} $, with $ \lambda\in\mathbb{C} $.

Fix $\bm\kappa\in(\C^*)^h$ and $\eu M\in\frak M_{s,h}$ as above.

\begin{defn}
We define the \emph{Borel $\eu M$-multitransform} as the $A$-linear morphism
\[\mathscr B_{\eu M}\colon \bigotimes_{j=1}^h\mathscr F_{\kappa_j}(A)\to \mathscr F_{\,\eu M\wedge\bm\kappa}(A),
\]which is defined, on decomposable elements, by
\[\mathscr B_{\eu M}\left(\bigotimes_{j=1}^h\left(\sum_{c_j\in M_{A,\kappa_j}}f_{c_j}^jZ^{c_j}\right)\right):=\sum_{\substack{c_j\in M_{A,\kappa_j}\\ j=1,\dots, h}}\frac{\prod_{a=1}^hf^a_{c_a}}{\prod_{i=1}^s\Gamma\left(1+\sum_{\ell=1}^h\iota_{ \kappa_\ell}(c_\ell)\beta_\ell^i\right)}Z^{\rho_{\bm \kappa}\left(\oplus_{\ell=1}^h \frac{c_\ell}{\al_\ell\sum_{i=1}^s\bt^i_\ell }\right)}.
\]
\end{defn}

\begin{defn}
We define the \emph{Laplace $\eu M$-multitransform} as the $A$-linear morphism
\[\mathscr L_{\eu M}\colon \bigotimes_{j=1}^h\mathscr F_{\kappa_j}(A)\to \mathscr F_{\,\eu M\vee\bm\kappa}(A),
\]which is defined, on decomposable elements, by
\begin{multline*}\mathscr L_{\eu M}\left(\bigotimes_{j=1}^h\left(\sum_{c_j\in M_{A,\kappa_j}}f_{c_j}^jZ^{c_j}\right)\right):=\\\sum_{\substack{c_j\in M_{A,\kappa_j}\\ j=1,\dots, h}}\left(\prod_{a=1}^hf^a_{c_a}\right)\prod_{i=1}^s\Gamma\left(1+\sum_{\ell=1}^h\iota_{ \kappa_\ell}(c_\ell)\beta_\ell^i\right)Z^{\rho_{\bm \kappa}(\oplus_{\ell=1}^h c_\ell\al_\ell\sum_{i=1}^s\bt^i_\ell )}.
\end{multline*}
\end{defn}

\subsection{Analytification of elements of \texorpdfstring{$\mathscr F_{\bm\kappa}(A)$}{}} 

Let $s=((\kappa_in_i1_A)_{i=1}^h,r)\in M_{A,\bm\kappa}$. We define the {\it analytification} $\widehat{Z^s}$ of the monomial $Z^s\in\mathscr F_{\bm \kappa}(A)$ to be the $A$-valued holomorphic function 
\[\widehat{Z^s}\colon \widetilde{\C^*}\to A,\quad \widehat{Z^s}(z):=z^{\sum_{i=1}^h\kappa_i n_i}\sum_{j=1}^\infty\frac{r^j}{j!}\log^j z.
\]Notice that the sum is finite, since $r\in{\rm Nil}(A)$. 

Given the series
\[f(Z)=\sum_{s\in M_{A,\bm\kappa}}f_sZ^s\in\mathscr F_{\bm\kappa}(A),
\]its {\it analytification} $\hat f$ is the $A$-valued holomorphic function defined by
\beq\label{fhat} \widehat f\colon W\subseteq \widetilde{\C^*}\to A,\quad \widehat f(z):=\sum_{s\in M_{A,\bm\kappa}}f_s\widehat{Z^s}(z),
\eeq provided the series on the r.h.s.\,\,is absolutely convergent. 
\begin{rem}
Notice that the series \eqref{fhat} has at most countably many non-zero terms, by Lemma \ref{supplem}.
\end{rem}

Fix $\bm\kappa\in(\C^*)^h$ and $\eu M\in\frak M_{s,h}$ as in the previous section.

\begin{thm}\label{BLAF}
Let $f_i\in \mathscr F_{ \kappa_i}(A)$, $i=1,\dots,h$, such that $\widehat f_i$ are well defined on $\R_+$. We have
\[\reallywidehat{\mathscr B_{\eu M}[\bigotimes_{j=1}^hf_j]}=\mathscr B_{\eu M}[\widehat f_1,\dots, \widehat f_h],
\]
\[\reallywidehat{\mathscr L_{\eu M}[\bigotimes_{j=1}^hf_j]}=\mathscr L_{\eu M}[\widehat f_1,\dots, \widehat f_h],
\]provided that both sides are well-defined.
\end{thm}
\proof 
It is sufficient to prove the statement on monomials $Z^{s_1},\dots, Z^{s_h}$. Let $s_j=(\kappa_j n_j 1_A,r_j)$ for $j=1,\dots, h$. We have
\beq\label{analeq1}
\reallywidehat{\mathscr B_{\eu M}[\otimes_{j=1}^h Z^{s_j}]}(z)=\frac{z^{\sum_{j=1}^h\frac{\kappa_j n_j}{\al_j\sum_{i=1}^s\bt^i_j}}}{\prod_{i=1}^s\Gamma\left(1+\sum_{\ell=1}^h(\kappa_\ell n_\ell 1_A+r_\ell)\beta_\ell^i\right)}\sum_{j=1}^\infty\frac{(\sum_{j=1}^h\frac{r_j}{\al_j\sum_{i=1}^s\bt_j^i}
)^j}{j!}\log^j z.
\eeq
A simple computation, similar to the one of \cite[Th.\,6.5.1]{Cot22}, gives
\begin{multline*}
\mathscr B_{\eu M}[\widehat{Z^{s_1}},\dots, \widehat{Z^{s_h}}](z) =z^{\sum_{j=1}^h\frac{\kappa_j n_j}{\al_j\sum_{i=1}^s\bt^i_j}}\\
\times\sum_{\substack{\ell_1,\dots,\ell_h\\w_1,\dots,w_h \\ u_1,\dots u_h=0}}^\infty\prod_{j=1}^h\prod_{i=1}^s\frac{r_j^{\ell_j}(\bt^i_j)^{u_j}}{w_j! u_j!}\left(\frac{\log z}{\al_j\sum_{a=1}^s\bt^a_j}\right)^{w_j}\left(\frac{1}{\Gamma}\right)^{(u_j)}\left(1+\sum_{\ell=1}^h\kappa_\ell n_\ell\bt_\ell^i\right)\delta_{w_j+u_j,\ell_j},
\end{multline*}
where $\dl_{a,b}$ is the Kronecker delta symbol.
This coincides with the formula of $\reallywidehat{\mathscr B_{\eu M}[\otimes_{j=1}^h Z^{s_j}]}(z)$, after expanding the $\frac{1}{\Gamma}$-factors in \eqref{analeq1} according to Definition \ref{formalF}. The proof for the Laplace multitransform is similar. 
\qed

\subsection{Proof of Theorems \ref{TH1} and \ref{mt2}}\label{secproof2}
Set $\varrho_j:=c_1(L_j)$ for $j=1,\dots,h$. 
\vskip1,5mm
By K\"unneth isomorphism, and by the universal property of coproduct of algebras (i.e. tensor product), we have injective  maps $H^\bullet(X_i,\C)\to H^\bullet(X,\C)$. {In particular, we have inclusions $\mathscr F_{\bm k}(X_j)\to \mathscr F_{\bm k}(X)$. } In order to ease the computations, in the next formulas we will not distinguish an element of $H^\bullet(X_i,\C)$ with its image in $H^\bullet(X,\C)$. So, for example, we will write \[c_1(X)=\sum_{j=1}^h\ell_h\varrho_j,\quad c_1(E)=\sum_{i=1}^sc_1(\mc L_i)=\sum_{i=1}^s\sum_{j=1}^hd_{ij}\varrho_j.\] The same will be applied for elements in $H_2(X,\Z)$.

The space of master functions $\mc S_{\bm\dl}(X)$ is generated by the components (with respect to an arbitrary basis of $\eu H_X$) of the small $J$-function
\[J_{X}\left(\bm\delta+c_1(X)\log z\right)\rqh=\bigotimes_{j=1}^h J_{X_j}(\delta_j+c_1(X_j)\log z),
\]by the R.\,Kaufmann's quantum K\"unneth formula \cite{Kau96}. By Lemma \ref{sjf}, for each $j=1,\dots,h$, we have
\begin{multline*}
\left.J_{X_j}(\delta_j+\log z\cdot c_1(X_j))\right\rqh\\=e^{\delta_j} z^{c_1(X_j)}\left(1+\sum_\al\sum_{\substack{\bt\in{\rm NE}(X_j)\\\bt\neq 0}}\sum_{k=0}^\infty e^{\int_\bt\delta_j} z^{\int_\bt c_1(X_j)}\langle\tau_k T_{\al,j},1\rangle_{0,2,\bt}^{X_j} T_j^\al\right),
\end{multline*}
where $(T_{0,j},\dots, T_{n_j,j})$ is a fixed basis of $\eu H_{X_j}$, 
$T^\al_j:=\sum_{\la=0}^{n_j}\eta^{\al\la}_jT_{\al,j}$, and $\eta_j$ is the Poincar\'e pairing on $\eu H_{X_j}$. 
So, we can rewrite the $J$-function $J_{X}\left(\bm\delta+c_1(X)\log z\right)\rqh$ in the form
\beq\label{jx2}
J_{X}\left(\bm\delta+c_1(X)\log z\right)\rqh=\bigotimes_{j=1}^h\sum_{k_j=0}^\infty J^j_{k_j}(\dl_j)z^{k_j\ell_j+\ell_j\varrho_j},
\eeq
where the coefficient $J^j_{k_j}(\dl_j)$ equals
\[J^j_{k_j}(\dl_j)=e^{\dl_j}\sum_{\al=0}^{n_j}\sum_{p=0}^\infty\sum_{\substack{\bt\in{\rm NE}(X_j)\\ \langle\varrho_j,\bt\rangle=k_j}}e^{\int_\bt\dl_j}\langle\tau_pT_{\al,j},1\rangle^{X_j}_{0,2,\bt}T_j^\al.
\]Each factor in the tensor product \eqref{jx2} is the analytification $\widehat{\rm J}_{X_i}$ of a series ${\rm J}_{X_j}\in\mathscr F_{\ell_j}(X)$ defined by
\beq
{\rm J}_{X_j}=\sum_{k_j=0}^\infty J^j_{k_j}(\dl_j)Z^{k_j\ell_j+c_1(X_j)},\quad j=1,\dots, h.
\eeq
\vskip1,5mm
\proof[Proof of Theorem \ref{TH1}. ] We apply the Quantum Lefschetz Theorem \ref{QLTH}. 
From equation \eqref{Ifun}, we deduce the formula
\begin{align}
\nonumber
&I_{X,Y}(\bm\delta+(c_1(X)-c_1(E))\log z)\rqh\\
\nonumber
&=\sum_{k_1,\dots,k_h\in\N}\bigotimes_{j=1}^h \left[J_{k_j}^j(\delta_j) z^{k_j(\ell_j-\sum_{i=1}^sd_{ij})+(\ell_j-\sum_{i=1}^sd_{ij})\varrho_j}\right]\prod_{i=1}^s \prod_{m=1}^{\sum_{p=1}^hd_{ip}k_p}\left(\sum_{p=1}^hd_{ip}\varrho_p+m\right) \\
\label{IXY2}
&=\sum_{k_1,\dots,k_h\in\N}\bigotimes_{j=1}^h \left[J_{k_j}^j(\delta_j) z^{k_j(\ell_j-\sum_{i=1}^sd_{ij})+(\ell_j-\sum_{i=1}^sd_{ij})\varrho_j}\right]\prod_{i=1}^s \frac{\Gamma(1+\sum_{p=1}^hd_{ip}k_p+\sum_{p=1}^hd_{ip}\varrho_p)}{\Gamma(1+\sum_{p=1}^hd_{ip}\varrho_p)}.
\end{align}
 The function in equation \eqref{IXY2} can be identified with the analytification of the Laplace $\eu M$-multitransform
\beq
\label{csq2}
{\rm I}_{X,Y}=\left(\prod_{i=1}^s\frac{1}{\Gamma(1+\sum_{p=1}^hd_{ip}\varrho_p)}\right)\cup_X{\mathscr L_{\eu M}[\otimes_{j=1}^h {\rm J}_{X_j}]},\quad 
\eeq
where
\[\eu M=\begin{pmatrix}
\al_1&\dots&\al_h\\
\bt^1_1&&\bt^1_h\\
\vdots&\ddots&\vdots\\
\bt^s_1&\dots&\bt^s_h
\end{pmatrix},\qquad \al_j=\frac{\ell_j-\sum_{i=1}^sd_{ij}}{\sum_{i=1}^sd_{ij}},\qquad \bt^i_j=\frac{d_{ij}}{\ell_j},
\]for $i=1,\dots,s$ and $j=1,\dots,h$. 
 The series ${\rm I}_{X,Y}$ can be seen as an element of $\mathscr F_{\bm \kappa}(X)$, with $\bm \kappa=\eu{M}\vee(\ell_1,\dots,\ell_h)=(\ell_j-\sum_{i=1}^sd_{ij})_{j=1}^h$, via the K\"unneth isomorphism. 
 
 By Theorem \ref{QLTH} and Theorem \ref{BLAF}, we have
\[J_Y(\iota^*\bm\delta+c_1(Y)\log z)\rqh=\iota^*\widehat{\rm I}_{X,Y}(\bm \delta+(c_1(X)-c_1(E)))\exp(-zH(\bm \delta)|_{\bf Q=1}).
\]
Thus, the components of the r.h.s., with respect to any basis of $H^\bullet(Y,\C)$, span the space of master functions $\mc S_{\iota^*\delta}(Y)$,  by Theorem \ref{masJ}. The factor $\prod_{i=1}^s{\Gamma(1+\sum_{p=1}^hd_{ip}\varrho_p})^{-1}$ coming from \eqref{csq2} can be eliminated by a change of basis of $H^\bullet(Y,\C)$. The claim follows by setting $c_{\bm\delta}:=H(\bm \delta)|_{\bf Q=1}$. \qed

\proof[Proof of Theorem \ref{mt2}] From the exact sequence 
\[0\to T_{P/X}\to TP\to \pi^*TX\to 0,
\]and the Euler exact sequence \cite[Sec.\,11.1.2]{EH16}
\[0\to\eu O_P\to \pi^*\left(\eu O_X\oplus E\right)\otimes \eu O_P(1)\to T_{P/X}\to 0,
\]we deduce
\[c_1(P)=\pi^*c_1(X)+\pi^*c_1(E)+(s+1)\xi.
\]
By Elezi--Brown Theorem \ref{EBT}, we have
\begin{multline*}
J_P(\pi^*\bm \dl+c_1(P)\log z)\rqh\\
=I_P\left[\pi^*\left(\bm\dl+\log z(c_1(X)+\sum_{i=1}^sc_1\left(\boxtimes_{j=1}^hL_j^{\otimes(-d_{ij})}\right)\right)+(s+1)\xi \log z\right]\rqh\\
=\sum_{\nu,k_1,\dots,k_h\in\N}\pi^*\left[\bigotimes_{j=1}^h J^j_{k_j}\left(\dl_j+\log z(\ell_j-\sum_{i=1}^sd_{ij})\varrho_j\right) \right]\mc T_{\nu,k_1,\dots,k_h}z^{(s+1)\nu+(s+1)\xi}\\
=\sum_{\nu,k_1,\dots,k_h\in\N}\pi^*\left[\bigotimes_{j=1}^h J^j_{k_j}(\dl_j)z^{(\ell_j-\sum_{i=1}^sd_{ij})(k_j+\varrho_j)} \right]\mc T_{\nu,k_1,\dots,k_h}z^{(s+1)\nu+(s+1)\xi},
\end{multline*}
where
\begin{multline*}\mc T_{\nu,k_1,\dots,k_h}=\frac{1}{\prod_{m=1}^\nu(\xi+m)}\prod_{i=1}^s\frac{1}{\prod_{m=1}^{\nu-\sum_{j=1}^hk_jd_{ij}}(\xi-\sum_{j=1}^hd_{ij}\varrho_j+m)}\\=\frac{\Gm(1+\xi)}{\Gm(1+\nu+\xi)}\prod_{i=1}^s\frac{\Gm(1+\xi-\sum_{j=1}^hd_{ij}\varrho_j)}{\Gm(1+\nu+\xi-\sum_{j=1}^hd_{ij}(k_j+\varrho_j))}.
\end{multline*}
Introduce the Ribenboim series $\mc E_P(\xi;Z)\in\mathscr F_1(P)$ defined by
\[\mc E_P(\xi;Z):=\sum_{k=0}^\infty\frac{Z^{k+\xi}}{\Gm(1+k+\xi)}.
\]Its analytification $\Hat{\mc E_P}(\xi;z)$ equals
\beq\label{analEP}
\Hat{\mc E_P}(\xi;z)=\sum_{k=0}^\infty\frac{1}{k!}\mc E_k(z)\xi^k.
\eeq
Given an arbitrary matrix $\eu M$ as in \eqref{Matrix}, the Borel mutlitransform $\mathscr B_{\eu M}[\pi^*{\rm J}_{X_1},\dots,\pi^*{\rm J}_{X_h},\mc E_P]$ equals
\begin{multline*}
\mathscr B_{\eu M}\left[\pi^*{\rm J}_{X_1},\dots,\pi^*{\rm J}_{X_h},\mc E_P\right]\\=
\sum_{k_1,\dots,k_h=0}^\infty\sum_{k=0}^\infty\frac{\pi^*\left(\bigotimes_{j=1}^h J^j_{k_j}\left(\dl_j\right) \right)}{\Gm(1+k+\xi)}\frac{Z^{\sum_{j=1}^h\frac{k_j\ell_j+c_1(X_j)}{\al_j\sum_{i=1}^s\bt^i_j}+\frac{k+\xi}{\al_{h+1}\sum_{i=1}^s\bt^i_{h+1}}}}{\prod_{i=1}^s\Gm(1+\sum_{j=1}^h\bt_j^i(k_j\ell_j+\ell_j\rho_j)+\bt_{h+1}^i(k+\xi))}.
\end{multline*}
Consequently, for the choice of weights
\[\al_j=\frac{\ell_j^2}{(\sum_{i=1}^sd_{ij})[(\sum_{i=1}^sd_{ij})-\ell_j]},\quad j=1,\dots,h,\qquad \al_{h+1}=\frac{1}{s(s+1)},
\]
\[\bt^i_j=-\frac{d_{ij}}{\ell_j},\quad i=1,\dots,s,\quad j=1,\dots,h,\qquad \bt^i_{h+1}=1,\quad i=1,\dots,s,
\]we have
\[J_P(\pi^*\bm \dl+c_1(P)\log z)\rqh=\Gm(1+\xi)\left[\prod_{i=1}^s\Gm(1+\xi-\sum_{j=1}^hd_{ij}\varrho_j)\right]\reallywidehat{\mathscr B_{\eu M}\left[\pi^*{\rm J}_{X_1},\dots,\pi^*{\rm J}_{X_h},\mc E_P\right]}.
\]
Then, the space $\mc S_{\pi^*\bm\dl}(P)$ is spanned by the components, with respect to an arbitrary basis of $H^\bullet(P,\C)$, of the analytification $\reallywidehat{\mathscr B_{\bm\al,\bm\bt}\left[\pi^*{\rm J}_{X_1},\dots,\pi^*{\rm J}_{X_h},\mc E_P\right]}$. This follows from the invertibility of  the morphism
\[H^\bullet(P,\C)\to H^\bullet(P,\C),\quad v\mapsto \Gm(1+\xi)\left[\prod_{i=1}^s\Gm(1+\xi-\sum_{j=1}^hd_{ij}\varrho_j)\right]v.
\]The statement of Theorem \ref{mt2} then follows from Theorem \ref{BLAF} and equation \ref{analEP}.
\qed

\section{Examples and applications}\label{sec6}

\subsection{Projective spaces} The projective space $\Pb^{n-1}$ is the toric variety associated with the weight data ${\sf T}^r=\C^*$, $u_1,\dots,u_n=1\in\Z\cong\Hom(T^r,\C^*)$ and $\om=1\in\R\cong\Hom(T^r,\C^*)\otimes\R$. The vector $T_1=1\in\Z\cong\Hom(T^r,\C^*)$ can be chosen an a nef integral basis, so that the weight data can be codified in the matrix ${\sf M}=(1,\dots,1)$.

The space of master functions $\mc S_0(\Pb^{n-1})$ equals the space of solutions $\Phi\in\eu O(\Tilde{\C^*})$ of the scalar quantum differential equation
\beq\label{qdepn}
\thi_z^n\Phi(z)=(nz)^n\Phi(z),\quad \thi_z:=z\frac{d}{dz}.
\eeq This equation has been extensively studied in \cite{Guz99,CDG24}.

Consider the generalized Mittag--Leffler function 
\begin{multline*}
\mc E_{\sf M}(s,z)=\frac{1}{2\pi n \sqrt{-1}}\int_{\frak H}\frac{\Gm(t/n)\Gm(1-t/n)}{\Gm(1+s-t/n)^n}e^{-\pi\sqrt{-1}t/n}z^{-t+ns}{\rm d}t=\sum_{k=0}^\infty\frac{z^{kn+ns}}{\Gm(1+k+s)^n},\\ (s,z)\in\C\times\Tilde{\C^*},
\end{multline*}
where $\frak H$ is a Hankel contour encircling the poles of the factor $\Gm(t/n)$ in the positive direction.

Theorem \ref{MTH0} implies that the functions $\mc E_{{\sf M},(k)}$, with $k=0,\dots,n-1$, are a basis of $\mc S_0(\Pb^{n-1})$. This can also easily checked by a direct computation.
\begin{lem}
We have $\thi_z^n\mc E_{\sf M}(s,z)=(nz)^n\mc E_{\sf M}(s,z)+(nz^s\Gm(s)^{-1})^n$. Hence, the functions $\mc E_{{\sf M},(k)}(z)$, with $k=0,\dots,n-1$, are solutions of \eqref{qdepn}.
\end{lem}
\proof We have
\[
\thi_z^n\mc E_{\sf M}(s,z)=\sum_{k=0}^\infty\frac{n^n(k+s)^nz^{kn+ns}}{\Gm(1+k+s)^n}=n^n\sum_{k=0}^\infty\frac{z^{kn+ns}}{\Gm(k+s)^n}=\frac{n^nz^{ns}}{\Gm(s)^n}+n^nz^n\mc E_{\sf M}(s,z).
\]Let us now take the derivatives $\der^k_s|_{s=0}$ of both sides. Since $\Gm(s)\sim\frac{1}{s}-\gm_{\rm EM}+O(s)$ for $s\to 0$, we have $(nz^s\Gm(s)^{-1})^n\sim n^ns^n(1+ns\log z+O(s^2))$ for $s\to 0$. Hence, $\der^k_s|_{s=0}(nz^s\Gm(s)^{-1})^n=0$ for $k=0,\dots,n-1$. The claim follows.
\endproof

If we introduce the 
complete Bell polynomials $ B_n(x_1, \dots, x_n) $ via the exponential generating function
\[
\sum_{n=0}^{\infty} B_n(x_1, \dots, x_n) \frac{t^n}{n!} = \exp\left( \sum_{j=1}^{\infty} x_j \frac{t^j}{j!} \right),
\]we can write explicit integral representations 
\beq\label{emkpn}
\mc E_{{\sf M},(k)}(z)=\left.\frac{\der^k}{\der s^k}\right|_{s=0}\mc E_{\sf M}(s,z)=\frac{1}{2n\sqrt{-1}}\int_{\frak H}\frac{\La_k(z,t)}{\sin(\pi t/n)}e^{-\pi\sqrt{-1}t/n}z^{-t}{\rm d}t,
\eeq where
\[\La_k(z,t):=\frac{B_k\left(n\log z-n\psi(1-t/n), -n\psi^{(1)}(1-t/n),\quad \dots,-n\psi^{(k-1)}(1-t/n)\right)}{\Gamma\left(1 - t/n\right)^n},
\]for any $k\geq 0$, and where $\psi(x):=\Gm'(x)/\Gm(x)$ is the digamma function. For $k=0$, the integral \eqref{emkpn} simplifies to
\[
\mc E_{{\sf M},(0)}(z)=\frac{1}{2\pi n\sqrt{-1}}\int_{\frak H}\Gm(t/n)^n\frac{\sin(\pi t/n)^{n-1}}{\pi^{n-1}}e^{-\pi\sqrt{-1}t/n}z^{-t}{\rm d}t=\sum_{d=0}^\infty\frac{z^{nd}}{(d!)^n},
\]coinciding with the quantum period $G_{\Pb^{n-1}}(z)$ of $\Pb^{n-1}$.

As an alternative to the basis $\left(\mc E_{{\sf M},(k)}(z)\right)_{k=0}^{n-1}$, we can exploit the symmetries of the differential equation \eqref{qdepn} to construct more bases of solutions. Indeed, if $\Phi(z)$ is an element of $\mc S_0(\Pb^{n-1})$, then also $\Phi(e^{2\pi\sqrt{-1}/n}z)$ is an element of $\mc S_0(\Pb^{n-1})$. Hence, for example, also the functions $\mc E_{{\sf M},(0)}(e^{2\pi\sqrt{-1}k/n}z)$, with $k=0,\dots, n-1$, define a basis of $\mc S_0(\Pb^{n-1})$. 

\subsection{Fano toric complete intersections} By combining Theorems \ref{MTH0} and \ref{TH1}, one obtains an explicit description of the space of master functions for complete intersections in products of Fano toric varieties.

For simplicity of notation, we focus on the case of a single Fano toric variety $X$, determined by the data $(r, N, \mathsf{M}, \omega)$ as introduced in Sections~\ref{sectoric} and~\ref{secmstoric}.  Let $L\in{\rm Pic}(X)$ be an ample line bundle such that $\det TX=L^{\otimes \ell}$, with $\ell\in\N_{>0}$, and $Y$ be a complete intersection in $X$ defined as the zero locus of a regular section of $\bigoplus_{i=1}^RL^{\otimes d_i}$, with $d_i\in\N_{>0}$ such that $\sum_{i=1}^Rd_i<\ell$.

Then, any master function of $Y$ at the point $0\in Q\!H(Y)$ is a $\C$-linear combination of functions of the form $e^{-cz}\mathscr L_{\eu M}[\mc E_{{\sf M},\bm\al}](z)$, where 
\[c\in\Q,\qquad\eu M=\left(\frac{\ell-\sum_{i=1}^Rd_i}{\sum_{i=1}^Rd_i},\frac{d_1}{\ell},\dots,\frac{d_R}{\ell}\right)^T,\qquad  \bm\al\in \N^r,\,|\bm\al|=0,\dots, N-r.\]In other words, any master function of $Y$ at $0\in Q\!H(Y)$ is a linear combination of integrals of the form
\[\frac{e^{-cz}}{(2\pi\sqrt{-1})^r}\int_{\frak H}\left.\der^{\bm\al}_{\bm s}\Theta_{{\sf M},\ell,\bm d}(\bm \zeta,\bm s, z)\right|_{\bm s=0}{\rm d}\zeta_1\dots{\rm d}\zeta_r,
\]where $(\bm s,z)\in\C^r\times\Tilde{\C^*}$, $\bm\al\in\N^r,\,|\bm\al|=0,\dots,N-r$, and
\begin{multline*}
\Theta_{{\sf M},\ell,\bm d}(\bm \zeta,\bm s, z):=\frac{1}{\prod_{a=1}^r\sum_{i=1}^Nm^a_i}\cdot\frac{\prod_{a=1}^r\Gm\left(\frac{\zeta_a}{\sum_{i=1}^Nm^a_i}\right)\Gm\left(1-\frac{\zeta_a}{\sum_{i=1}^Nm^a_i}\right)}{\prod_{i=1}^N\Gm\left(1-\sum_{a=1}^r\frac{m^a_i}{\sum_{j=1}^Nm^a_j}\zeta_a+\sum_{a=1}^rm^a_is_a\right)}\\
\times\prod_{b=1}^R\Gm\left(1-\frac{d_b}{\ell}\sum_{a=1}^r\zeta_a+\frac{d_b}{\ell}\sum_{a=1}^r\sum_{i=1}^Nm^a_is_a\right)
\\
\times \exp\left({-\pi\sqrt{-1}\sum_{a=1}^r\frac{\zeta_a}{\sum_{i=1}^Nm^a_i}}\right)z^{(-\sum_{a=1}^r\zeta_a+\sum_{a=1}^r\sum_{i=1}^Nm^a_is_a)(\ell-\sum_{i=1}^Rd_i)/\ell},
\end{multline*}
and where $\frak H$ is the products of Hankel contours encircling, in the positive directions, the poles of the factors $\Gm(\zeta_a/\sum_{i=1}^Nm^a_i)$, with $a=1,\dots,r$.

\subsection{Fano bundles on toric projective varieties} By combining Theorems \ref{MTH0} and \ref{mt2}, one obtains an explicit description of the space of master functions for Fano bundles over products of toric varieties (each automatically Fano \cite[Thm.\,1.6]{SW90}).

As an example, let us focus on the case of a single Fano toric variety $X$, determined by the data $(r, N, \mathsf{M}, \omega)$ as introduced in Sections~\ref{sectoric} and~\ref{secmstoric}. Let $L\in{\rm Pic}(X)$ be an ample line bundle such that $\det TX=L^{\otimes \ell}$, with $\ell\in\N_{>0}$, and $P$ be the total space of the bundle $\Pb\left(\eu O_X\oplus\bigoplus_{i=1}^{R}L^{\otimes (-d_i)}\right)$, with $d_i\in\N_{>0}$ such that $\sum_{i=1}^{R}d_i<\ell$.

Then, any master function of $P$ at the point  $0\in Q\!H(P)$ is a $\C$-linear combination of functions of the form $\mathscr B_{\eu M}(\mc E_{{\sf M},\bm\al},\mc E_k)$, where
\[\eu M=\begin{pmatrix}
\frac{\ell^2}{(\sum_{i=1}^{R}d_i)((\sum_{i=1}^{R}d_i)-\ell)}&\frac{1}{{R}({R}+1)}\\
-d_1/\ell&1\\
-d_2/\ell&1\\
\vdots&\vdots\\
-d_{R}/\ell&1
\end{pmatrix},\qquad\bm\al\in\N^r,\,|\bm\al|=0,\dots,N-r+{R},\quad k=0,\dots,N-r. 
\]In other words, any master function of $P$ at $0\in Q\!H(P)$ is a linear combination of integrals of the form
\[\frac{1}{(2\pi\sqrt{-1})^{r+1}}\int_{\frak H}\left.\der^{\bm\al}_{\bm s}\Om_{{\sf M},\ell,\bm d}(\bm\zeta,\bm s,z)\right|_{\bm s=0}{\rm d}\zeta_1\dots{\rm d}\zeta_{r+1},
\]where $(\bm s,z)\in\C^{r+1}\times\Tilde{\C^*}$, $\bm\al\in\N^{r+1},\,|\bm\al|=0,\dots,2(N-r)+R$, and

\begin{multline*}
\Om_{{\sf M},\ell,\bm d}(\bm\zeta,\bm s,z)=\frac{1}{\prod_{a=1}^r\sum_{i=1}^Nm^a_i}\cdot\frac{\prod_{a=1}^r\Gm\left(\frac{\zeta_a}{\sum_{i=1}^Nm^a_i}\right)\Gm\left(1-\frac{\zeta_a}{\sum_{i=1}^Nm^a_i}\right)}{\prod_{i=1}^N\Gm\left(1-\sum_{a=1}^r\frac{m^a_i}{\sum_{j=1}^Nm^a_j}\zeta_a+\sum_{a=1}^rm^a_is_a\right)}\\
\times\frac{\Gm(\zeta_{r+1})\Gm(1-\zeta_{r+1})}{\Gm(1-\zeta_{r+1}+s_{r+1})}\prod_{b=1}^R\Gm\left(1-\zeta_{r+1}+s_{r+1}+\frac{d_b}{\ell}\sum_{a=1}^r\zeta_a-\frac{d_b}{\ell}\sum_{a=1}^r\sum_{i=1}^Nm^a_is_a\right)^{-1}\\
\times \exp\left[-\pi\sqrt{-1}\left(\zeta_{r+1}+\sum_{a=1}^r\frac{\zeta_a}{\sum_{i=1}^Nm^a_i}\right)\right]z^{(-\sum_{a=1}^r\zeta_a+\sum_{a=1}^r\sum_{i=1}^Nm^a_is_a)\frac{\ell-\sum_{i=1}^Rd_i}{\ell}+R(s_{r+1-\zeta_{r+1}})},
\end{multline*}and where $\frak H$ is the products of Hankel contours encircling, in the positive directions, the poles of the factors $\Gm(\zeta_a/\sum_{i=1}^Nm^a_i)$, with $a=1,\dots,r$, and $\Gm(\zeta_{r+1})$.

\subsection{Strict del Pezzo surfaces}\label{sdP} Two-dimensional Fano varieties, known as {\it del Pezzo surfaces}, fall into exactly ten isomorphism classes: $\mathbb{P}^2$, $\mathbb{P}^1 \times \mathbb{P}^1$, and the blow-ups of $\mathbb{P}^2$ at $r$ points in general position, for $1 \leq r \leq 8$, denoted ${\rm dP}_r$. Among these, the {\it strict del Pezzo surfaces} are those with very ample anticanonical divisor: $\mathbb{P}^2$, $\mathbb{P}^1 \times \mathbb{P}^1$, and ${\rm dP}_r$ for $1 \leq r \leq 6$.

The results of this paper provide explicit descriptions of the master function spaces for all such varieties, via Mellin--Barnes integral representations of solutions to their quantum differential equations.

Strict del Pezzo surfaces admit the following geometric constructions, see \cite[Ch.\,3.2]{IP99}, \cite{Cor02,CCGK16,HKLT21}:
\begin{itemize}
\item $\mathbb{P}^2$, $\mathbb{P}^1 \times \mathbb{P}^1$, and ${\rm dP}_r$ for $r = 1, 2, 3$ are toric varieties;
\item ${\rm dP}_1$ can also be realized as a hypersurface of bidegree $(1,1)$ in $\mathbb{P}^1 \times \mathbb{P}^2$, and as the total space of the Fano projective bundle $\mathbb{P}(\eu{O} \oplus \eu{O}(-1))$ over $\mathbb{P}^1$;
\item ${\rm dP}_2$ can also be described as the complete intersection of divisors of degrees $(1,0,1)$ and $(0,1,1)$ in $\mathbb{P}^1 \times \mathbb{P}^1 \times \mathbb{P}^2$;
\item ${\rm dP}_3$ is also the complete intersection of two divisors of bidegree $(1,1)$ in $\mathbb{P}^2 \times \mathbb{P}^2$;
\item ${\rm dP}_4$ is a hypersurface of bidegree $(1,2)$ in $\mathbb{P}^1 \times \mathbb{P}^2$;
\item ${\rm dP}_5$ is a complete intersection of two quadrics in $\mathbb{P}^4$, that is, of type $(2,2)$;
\item ${\rm dP}_6$ is a smooth cubic surface in $\mathbb{P}^3$.
\end{itemize}

For each of these realizations, Theorems \ref{MTH0}, \ref{TH1}, and \ref{mt2} provide explicit integral solutions to the quantum differential equations of all strict del Pezzo surfaces. These representations are particularly well suited for analyzing asymptotics, the Stokes phenomenon, and other analytic aspects. Further details will be presented in a forthcoming publication. Here, we focus on the examples of the del Pezzo surfaces ${\rm dP}_1$ and ${\rm dP}_2$, and explicitly integrate their quantum differential equations in terms of generalized Mittag--Leffler functions.

\begin{example} Consider the del Pezzo surface ${\rm dP}_1$, the blow-up of $\mathbb{P}^2$ at a point. Let $T_1, T_2 \in H^2({\rm dP}_1, \mathbb{Z})$ be the cohomology classes where $T_1$ is the strict transform of a line through the blown-up point, and $T_2$ is the pullback of the hyperplane class of $\mathbb{P}^2$. In terms of the exceptional divisor $E$ and hyperplane class $H$, we have $T_1 = H - E$ and $T_2 = H$. The pair $(T_1, T_2)$ forms a nef integral basis of $H^2({\rm dP}_1, \mathbb{Z})$.

The surface ${\rm dP}_1$ also admits a toric realization associated with the following data (in the notation of Sections~\ref{sectoric} and~\ref{secmstoric}): $r=2$, $N=4$, weight data matrix ${\sf M}=\begin{pmatrix}
1&1&-1&0\\
0&0&1&1
\end{pmatrix}$, computed with respect with the nef basis $(T_1,T_2)$, and stability condition $\om=(1,1)$. In other words, ${\rm dP}_1$ is realized as the GIT quotient ${\rm dP}_1=\C^4/\!\!/_{\om}(\C^*)^2=\mc U_{\om}/(\C^*)^2$, where $\mc U_\om=\C^4\setminus(\{x_1=x_2=0\}\cup\{x_3=x_4=0\})$, and the torus $(\C^*)^2$ acts on $\C^4$ via $(t_1,t_2)\cdot(x_1,x_2,x_3,x_4)=(t_1x_1,t_1x_2,t_1^{-1}t_2x_3,t_2x_4)$.

The space $\mc S_0({\rm dP}_1)$ of master functions at $0\in Q\!H({\rm dP}_1)$ is the space of $\Phi\in\eu O(\Tilde{\C^*})$ such that $L\Phi=0$, where\footnote{\,\,Warning: There is a typo in \cite[formula (11.2.1)]{Cot22} and \cite[formula (5.2)]{Cot24}, where a factor of $z$ multiplying $\thi_z^2$ was omitted. Nevertheless, the computations in {\it loc.\,cit.} are still correct, as they were carried out using the correct expression of the differential operator $L$.} 
\begin{align}\label{qdiffdisp2bis}
L=&(283 z-24)\vartheta_z^4+\left(283 z^2-590 z+24\right)\vartheta_z^3+ z\left(-2264 z^2+192 z+3\right)\vartheta_z^2\\
\nonumber&-4 z^2 \left(2547 z^2+350 z-104\right)\vartheta_z+z^2 \left(-3113 z^3-9924 z^2+1476 z+192\right)=0,
\end{align}and $\thi_z:=z\frac{d}{dz}$.

\begin{rem}
The quantum differential equation \eqref{qdiffdisp2bis} has been thoroughly studied in \cite{Cot22,Cot24}. In \cite{Cot22}, solution bases were constructed via Laplace multitransforms, using the realization of ${\rm dP}_1$ as a complete intersection in $\mathbb{P}^1 \times \mathbb{P}^2$. In \cite{Cot24}, Borel multitransforms were used, based on its structure as a Fano bundle over $\mathbb{P}^1$.
\end{rem}

Introduce the generalized Mittag--Leffler function
\begin{multline*}\mc E_{\sf M}(s_1,s_2,z)
=\frac{1}{(2\pi\sqrt{-1})^2}\\\times\int\!\!\!\!\int_{\frak H_1\times \frak H_2}\frac{\exp \left[-\pi  \sqrt{-1} \left(t_1+t_2/2\right)\right] \Gamma (t_1) \Gamma (1-t_1) \Gamma \left(t_2/2\right) \Gamma \left(1-t_2/2\right) z^{s_1+2 s_2-t_1-t_2}}{2 \Gamma (s_1-t_1+1)^2 \Gamma \left(s_2-t_2/2+1\right) \Gamma \left(-s_1+s_2+t_1-t_2/2+1\right)}{\rm d}t_1{\rm d}t_2\\=\sum_{d_1,d_2=0}^\infty\frac{z^{d_1+2d_2+s_1+2s_2}}{\Gm(1+d_1+s_1)^2\Gm(1-d_1+d_2-s_1+s_2)\Gm(1+d_2+s_2)},
\end{multline*}where $\frak H_1$ (resp.\,$\frak H_2$) is a Hankel contour in the $t_1$-plane (resp.\,$t_2$-plane) encircling the poles of $\Gm(t_1)$ (resp.\,$\Gm(t_2/2)$) in the positive direction.

\begin{prop}
The functions 
\[\mc E_{{\sf M},(0,0)}(z),\quad \mc E_{{\sf M},(1,0)}(z),\quad \mc E_{{\sf M},(0,1)}(z),\quad 4\mc E_{{\sf M},(1,1)}(z)+\mc E_{{\sf M},(0,2)}(z),\] define a basis of $\mc S_0({\rm dP}_1)$.
\end{prop}
\proof
In the classical cohomology ring, we have the relations $T_1^2=0$, and $T_1T_2=T_2^2=[{\rm pt}]$. The space of master functions is spanned by the components of the $J$-function 
\[
J_{{\rm dP}_1}(c_1({\rm dP}_1)\log z)\rqh = \widehat\Gm^+_{{\rm dP}_1}\cup\mc E_{\sf M}(T_1,T_2,z)
\] w.r.t.\,\,an arbitrary basis of the cohomology of ${\rm dP}_1$.  Since $\widehat\Gm^+_{{\rm dP}_1}=\prod_{i=1}^4\Gm(1+u_i)=\Gm(1+T_1)^2\Gm(1-T_1+T_2)\Gm(1+T_2)=1+\dots$, the operator of $\cup$-multiplication by $\widehat\Gm^+_{{\rm dP}_1}$ is invertible. So, the space of master functions is spanned by the components of 
\begin{multline*}\mc E_{{\sf M}}(T_1,T_2,z)=1\cdot\mc E_{{\sf M},(0,0)}(z)+T_1\cdot \mc E_{{\sf M},(1,0)}(z)\\+ T_2\cdot \mc E_{{\sf M},(0,1)}(z)+T_1T_2\cdot \mc E_{{\sf M},(1,1)}(z)+T_2^2\cdot\frac{1}{4}\mc E_{{\sf M},(0,2)}(z),
\end{multline*}with respect to the basis $(1,T_1,T_2,[{\rm pt}])$.
\endproof
\end{example}

\begin{example}
Consider the del Pezzo surface $ {\rm dP}_2 $, the blow-up of $ \mathbb{P}^2 $ at two points $p_1$ and $p_2$. Let $ T_1, T_2, T_3 \in H^2({\rm dP}_2, \mathbb{Z}) $ be the cohomology classes where:
\begin{itemize}
    \item $ T_1 $ is the strict transform of a line passing through $p_1$ but not $p_2$,
    \item $ T_2 $ is the strict transform of a line passing through $p_2$ but not $p_1$,
    \item $ T_3 $ is the pullback of the hyperplane class from $ \mathbb{P}^2 $.
\end{itemize}
In terms of the exceptional divisors $ E_1, E_2 $ and the hyperplane class $ H $, we have:
\[
T_1 = H - E_1 , \quad T_2 = H - E_2, \quad T_3 = H.
\]
The triple $ (T_1, T_2, T_3) $ forms a nef integral basis of $ H^2({\rm dP}_2, \mathbb{Z}) $.

The surface $ {\rm dP}_2 $ also admits a toric realization associated with the following data (in the notation of Sections~\ref{sectoric} and~\ref{secmstoric}): $ r = 3 $, $ N = 5 $, with weight data matrix
\[
\mathsf{M} =
\begin{pmatrix}
1 & 0 & 1 & -1 & 0 \\
0 & 1 & 1 & 0 & -1 \\
0 & 0 & -1 & 1 & 1
\end{pmatrix}
\]
computed with respect to the nef basis $ (T_1, T_2, T_3) $, and stability condition $ \omega = (1,1,1) $. In other words, $ {\rm dP}_2 $ is realized as the GIT quotient
\[
{\rm dP}_2 = \mathbb{C}^5 /\!\!/_{\omega} (\mathbb{C}^*)^3 = \mathcal{U}_{\omega}/(\mathbb{C}^*)^3,
\]
where
\[
\mathcal{U}_{\omega} = \mathbb{C}^5 \setminus \left( \{x_1 = x_4 = 0\} \cup \{x_2 = x_5 = 0\} \cup \{x_3 = 0\} \right),
\]
and the torus $ (\mathbb{C}^*)^3 $ acts on $ \mathbb{C}^5 $ via
\[
(t_1, t_2, t_3) \cdot (x_1, x_2, x_3, x_4, x_5) = (t_1 x_1, t_2 x_2, t_1 t_2 t_3^{-1} x_3, t_1^{-1} t_3 x_4, t_2^{-1}t_3 x_5).
\]

In the basis $(1,T_1,T_2,T_3,[{\rm pt}])$, the Poincar\'e pairing has Gram matrix
\[\eta=\left(
\begin{array}{ccccc}
 0 & 0 & 0 & 0 & 1 \\
 0 & 0 & 1 & 1 & 0 \\
 0 & 1 & 0 & 1 & 0 \\
 0 & 1 & 1 & 1 & 0 \\
 1 & 0 & 0 & 0 & 0 \\
\end{array}
\right),
\]as it is readily seen from the relations
\beq\label{relcoh}
T_1^2=T_2^2=0,\qquad T_1T_2=T_1T_3=T_2T_3=T_3^2=[{\rm pt}].
\eeq

In the basis $(1,T_1,T_2,T_3,[{\rm pt}])$ the operator of quantum multiplication $c_1({\rm dP}_2)\circ_0$, specialized at the origin $0\in QH({\rm dP}_2)$ equals
\[\mc U_0=\left(
\begin{array}{ccccc}
 0 & 2 & 2 & 4 & 3 \\
 1 & -1 & 0 & 1 & 2 \\
 1 & 0 & -1 & 1 & 2 \\
 1 & 1 & 1 & -1 & 0 \\
 0 & 2 & 2 & 3 & 0 \\
\end{array}
\right),
\]while the grading operator $\mu$ equals
\[
\mu={\rm diag}(-1,0,0,0,1).
\]The system of differential equations \eqref{qde.2} for a $\Hat\nabla$-flat section $\varpi$ is given by
\begin{align}
\label{qdedp2.or1}
\der_z\varpi_0&= \frac{\varpi_0}{z}+\varpi_1+\varpi_2+\varpi_3,\\
\der_z\varpi_1&=2 \varpi_0-\varpi_1+\varpi_3+2 \varpi_4,\\
\der_z\varpi_2&=2 \varpi_0-\varpi_2+\varpi_3+2 \varpi_4,\\
\der_z\varpi_3&=4 \varpi_0+\varpi_1+\varpi_2-\varpi_3+3 \varpi_4,\\
\label{qdedp2.or2}
\der_z\varpi_4&=3 \varpi_0+2 \varpi_1+2 \varpi_2-\frac{\varpi_4}{z}.
\end{align}

\begin{rem}
The operator $\mc U_0$ has characteristic polynomial
\[
\det(\mc U_0 - x\cdot \mathbf{1}) = -(x+1)^2 \left(x^3 + x^2 - 18x - 43\right),
\]
which exhibits a \emph{coalescence of canonical coordinates} at the point $0 \in QH({\rm dP}_2)$, a phenomenon thoroughly investigated in \cite{CG17,CG18,CDG19,CDG20}. 
A direct computation shows that the vectors $(e_0, \dots, e_4)$ defined in \eqref{eqvecek} are not linearly independent at this point. As a consequence, the reduction of system \eqref{qdedp2.or1} to a scalar quantum differential equation cannot proceed as outlined in Section~\ref{scyc}, and requires a more refined analysis, developed below.
In particular, the map $\nu_p$ from \eqref{eqnupmap} fails to be injective at $p=0$. Understanding the interplay between the coalescence phenomenon and the geometry of the $\mc A_{\La}$-stratum remains an open and intriguing problem. See also \cite[Sec.\,2.6]{Cot22}.
\end{rem}

The differential system \eqref{qdedp2.or1}-\eqref{qdedp2.or2} decomposes into the direct sum of a 1-dimensional and a 4-dimensional systems: if we set 
\beq\label{req1} \tilde{\varpi_0}=\varpi_0,\qquad \tilde\varpi_1=\varpi_1+\varpi_2,\qquad \tilde\varpi_2=\varpi_1-\varpi_2,\qquad \tilde\varpi_3=\varpi_3,\qquad \tilde\varpi_4=\varpi_4,
\eeq we obtain the decoupled systems
\begin{align}
\label{req2}
\der_z\tilde{\varpi}_2&=-\tilde\varpi_2\quad\Longrightarrow\quad \tilde\varpi_2=ce^{-z},\quad c\in\C,\\
\label{qdedp2.1}\der_z{\tilde\varpi}_0&= \frac{{\tilde\varpi}_0}{z}+{\tilde\varpi}_1+{\tilde\varpi}_3,\\
\der_z{\tilde\varpi}_1&=4 {\tilde\varpi}_0-{\tilde\varpi}_1+2{\tilde\varpi}_3+4 {\tilde\varpi}_4,\\
\der_z{\tilde\varpi}_3&=4 {\tilde\varpi}_0+{\tilde\varpi}_1+-{\tilde\varpi}_3+3 {\tilde\varpi}_4,\\
\label{qdedp2.4}\der_z{\tilde\varpi}_4&=-\frac{{\tilde\varpi}_4}{z}+3 {\tilde\varpi}_0+2 {\tilde\varpi}_1.
\end{align}
The master functions $\Phi\in\mc S_0({\rm dP}_2)$ are of the form
\[\Phi(z)=z^{-1}\varpi_0(z)=z^{-1}\tilde{\varpi}_0(z),
\]and the differential system \eqref{qdedp2.1}-\eqref{qdedp2.4} is equivalent to the scalar qDE
\begin{multline}\label{sqdedp2}
(115 z^2 - 7 z)\thi_z^4\Phi + 
 z (7 - 244  z + 230  z^2)\thi_Z^3\Phi + 
 z^2 (3 + 119  z - 1955  z^2) \thi_z^2\Phi\\
 +z^3 (280 - 1528  z - 7015  z^2) \thi_z\Phi+ 
 z^3  (112 + 578  z - 6714  z^2 - 4945  z^3) \Phi=0,
\end{multline}
where $\thi_z=z\frac{d}{dz}$. Given a solution $\Phi$ of \eqref{sqdedp2}, the solution $(\tilde\varpi_0,\tilde\varpi_1,\tilde\varpi_3,\tilde\varpi_4)$ of \eqref{qdedp2.1}-\eqref{qdedp2.4} can be reconstructed via the formulas
\begin{align}
\label{req3}
\tilde\varpi_0&=z\Phi(z),\\
\tilde\varpi_1&=  \frac{\left(-46 z^2-10 z+7\right) \Phi'(z)+z \left(7 z \Phi^{(3)}(z)-3 (z-7) \Phi''(z)\right)-z (151 z+112) \Phi(z)}{115 z-7},\\
\tilde\varpi_3&= \frac{\left(161 z^2+3 z-7\right) \Phi'(z)+z \left(3 (z-7) \Phi''(z)-7 z \Phi^{(3)}(z)\right)+z (151 z+112) \Phi(z)}{115 z-7},\\
\label{req4}
\tilde\varpi_4&=  -\frac{z \left(-z \Phi^{(3)}(z)-16 z \Phi''(z)-2 \Phi''(z)+(23 z-16) \Phi'(z)+(153 z+8) \Phi(z)\right)}{115 z-7}.
\end{align}
\begin{prop}
A solution $\varpi$ of the differential system \eqref{qdedp2.or1}-\eqref{qdedp2.or2} is uniquely determined by a pair $(\Phi,c)$, where $\Phi$ is a solution of the scalar qDE \eqref{sqdedp2} and $c\in\C$. Namely, the solution $\varpi$ can be reconstructed via the formulas \eqref{req1}, \eqref{req2}, \eqref{req3}-\eqref{req4}.\qedhere
\end{prop}

The scalar qDE \eqref{sqdedp2} can be integrated in terms of generalized Mittag--Leffler functions. Introduce the function
\begin{multline*}
\mc E_{\sf M}(s_1,s_2,s_3,z)=\frac{1}{(2\pi\sqrt{-1})^3}\times\\
\int_{\frak H}\frac{\prod_{i=1}^3\Gm(t_i)\Gm(1-t_i)e^{-\pi\sqrt{-1}(t_1+t_2+t_3)}z^{-(t_1+t_2+t_3)+s_1+s_2+s_3}}{\bm\Gm(\bm t,\bm s)}{\rm d}t_1{\rm d}t_2{\rm d}t_3\\
=\sum_{\bm d\in\N^3}\frac{z^{d_1+d_2+d_3+s_1+s_2+s_3}}{\bm\Gm(-\bm d,\bm s)},
\end{multline*}
where
\begin{multline*}\bm\Gm(\bm t,\bm s):=\Gm(1-t_1+s_1)\Gm(1-t_2+s_2)\Gm(1-t_1-t_2+t_3+s_1+s_2-s_3)\\
\times\Gm(1+t_1-t_3-s_1+s_3)\Gm(1+t_2-t_3-s_2+s_3),
\end{multline*}
and where $\frak H$ is the product of three Hankel contours in the $t_i$-plane, $i=1,2,3$, encircling the poles of $\Gm(t_i)$ in the positive directions.
\begin{rem}
Notice that $\mc E_{\sf M}(s_1,s_2,s_3,z)=\mc E_{\sf M}(s_2,s_1,s_3,z)$. This reflects the symmetric r\^oles of the blown-up points $p_1$ and $p_2$, and of the corresponding cohomology classes $T_1$ and $T_2$. We also deduce the identities
\beq\label{symEm}
\mc E_{{\sf M},(1,0,0)}(z)=\mc E_{{\sf M},(0,1,0)}(z),\qquad \mc E_{{\sf M},(1,0,1)}(z)=\mc E_{{\sf M},(0,1,1)}(z).
\eeq 
\end{rem}

\begin{prop}
The functions
\[
\mc E_{{\sf M},(0,0,0)}(z),\quad \mc E_{{\sf M},(1,0,0)}(z),\quad \mc E_{{\sf M},(0,0,1)}(z),\quad 4\mc E_{{\sf M},(1,1,0)}(z) + 8\mc E_{{\sf M},(1,0,1)}(z) + \mc E_{{\sf M},(0,0,2)}(z)
\]
form a basis of the space $\mc S_0({\rm dP}_2)$ of solutions to the scalar qDE~\eqref{sqdedp2}.
\end{prop}

\begin{proof}
Using the relations~\eqref{relcoh}, we have
\begin{multline*}
\mc E_{\sf M}(T_1,T_2,T_3,z) = 1\cdot \mc E_{{\sf M},(0,0,0)}(z) + T_1\cdot \mc E_{{\sf M},(1,0,0)}(z) + T_2\cdot \mc E_{{\sf M},(0,1,0)}(z) + T_3\cdot \mc E_{{\sf M},(0,0,1)}(z) \\
+ [{\rm pt}]\cdot \big( \mc E_{{\sf M},(1,1,0)}(z) + \mc E_{{\sf M},(1,0,1)}(z) + \mc E_{{\sf M},(0,1,1)}(z) + \tfrac{1}{4}\mc E_{{\sf M},(0,0,2)}(z) \big).
\end{multline*}
The space $\mc S_0({\rm dP}_2)$ is contained in the $\C$-span of the components of $\mc E_{\sf M}(T_1,T_2,T_3,z)$. The result then follows by applying the identities \eqref{symEm}.
\end{proof}
\end{example}

\end{document}